\documentclass[12pt,a4paper]{article}
\usepackage{amssymb}
\usepackage{srcltx}
\usepackage{amssymb}
\usepackage{enumerate}
\usepackage{multicol}
\usepackage{epstopdf}
\usepackage{lipsum}
\usepackage{textcomp}
\usepackage{marginnote}
\usepackage{graphicx}
\usepackage{subcaption}
\usepackage[section]{placeins}

\usepackage{amsfonts}
\usepackage{algorithmic}
\usepackage{float}
\ifpdf
  \DeclareGraphicsExtensions{.eps,.pdf,.png,.jpg}
\else
  \DeclareGraphicsExtensions{.eps}
\fi
\usepackage{enumitem}
\setlist[enumerate]{leftmargin=.5in}
\setlist[itemize]{leftmargin=.5in}

\RequirePackage[OT1]{fontenc}
\RequirePackage{amsthm,amsmath, xcolor}
\RequirePackage[numbers]{natbib}
\RequirePackage[colorlinks,citecolor=blue,urlcolor=blue]{hyperref}

 \usepackage[capitalise]{cleveref} 

\newtheorem{theorem}{Theorem}[section]
\newtheorem{proposition}[theorem]{Proposition}

\newtheorem{lemma}[theorem]{Lemma}
\newtheorem{definition}[theorem]{Definition}
\newtheorem{remark}{Remark}[section]

\newcommand\fdem{$\Box$}

\newcommand\cE{{\cal E}}
\newcommand\cH{{\cal H}}

\newcommand\cF{{\cal F}}
\newcommand\cL{{\cal L}}

\newcommand\cN{{\cal N}}
\newcommand\cM{{\cal M}}
\newcommand\cX{{\cal X}}

\newcommand\cV{{\cal V}}

\def\bbr{{\mathbb R}}

\def\text#1{\hbox{#1}}

\def\E{{\bf E}}

\def\P{{\bf P}}
\def\B{{\bf B}}

\def\C{{\bf C}}

\def\H{{\bf H}}

\def\c{{\bf c}}
\def\q{{\bf q}}
\def\u{{\bf u}}

\def\me{\mathrm{e}}

\def\d{\mathrm{d}}
\def\dsl{\displaystyle}

\def\build #1_#2{\mathrel{\mathop{\kern 0pt #1}\limits_{#2}}}

\newcommand{\zs}[1]{{\mathchoice{#1}{#1}{\lower.25ex\hbox{$\scriptstyle#1$}}
{\lower0.25ex\hbox{$\scriptscriptstyle#1$}}}}

\numberwithin{equation}{section}
\ifpdf
\hypersetup{
  pdftitle={Logarithmic utility case},
  pdfauthor={S. Albosaily, and S. Pergamenshchikov}
}
\fi

\begin{document}
\title{
Optimal investment and consumption for pairs trading financial markets on small time interval
}

\author{
Sahar Albosaily\thanks{
Laboratoire de Math\'ematiques Raphael Salem,
 UMR 6085 CNRS- Universit\'e de Rouen,  France 
  and 
 University of Hail, Saudi Arabia, 
 ORCID iD: 0000-0002-5714-7834,
 e-mail:sahar.albosaily@etu.univ-rouen.fr
}
  \, and \, 
 Serguei Pergamenshchikov\thanks{
 Laboratoire de Math\'ematiques Raphael Salem,
 UMR 6085 CNRS- Universit\'e de Rouen Normandie,  France
and
International Laboratory of Statistics of Stochastic Processes and
Quantitative Finance, National Research Tomsk State University,
 e-mail:
Serge.Pergamenshchikov@univ-rouen.fr } 
}

\date{}

\maketitle

\begin{abstract}
  In this paper we consider a pairs trading financial market with the spread of risky assets defined by the Ornstein-Uhlenbeck (OU) process. We implement an optimal strategy for power utility functions for investment/consumption problem. Through the Feynman-Kac (FK) method, we study the Hamilton-Jacobi-Bellman (HJB) equation for this problem. Moreover, the existence and uniqueness has been shown for classical solution for the HJB equation. In addition, the numeric approximation for the solution of the HJB equation has been studied and the convergence rate has been established and it is been found that the convergence rate is extremely explosive.
\end{abstract}

{\bf keywords}
 Optimality, 
Feynman--Kac mapping, 
Hamilton--Jacobi--Bellman equation, 
It\^o formula, 
Brownian motion, 
Ornstein--Uhlenbeck process,
Stochastic processes,
Financial market,
Spread market.


\section{Introduction}\label{sec:In}

This paper deals with an optimal investment/consumption problem during a fixed
time interval $[0,T]$ for a financial market generated by risky spread assets defined through the Ornstein--Uhlenbeck (OU) processes.
Such problems are of prime interest for  practical investors such as electricity gas markets.
 Also in other sectors like the microstructioe level within the airline industry (see, for example, 
 \cite{CaldeiraMoura2013}) as well it is known in many hedge funds \cite{CaldeiraMoura2013}.
%
Usually, for such model, one uses a dynamical programming method (see, for example, \cite{Karatzasshreve1998}.
%
 Unfortunately, we can not use the Hamilton--Jacobi--Bellman (HJB) analysis method developed for the Black--Scholes market since for the OU model in the HJB equation there is an additional variable corresponding to the risky asset. 
Note that in \cite{BoguslavskyBoguslavskaya2004} for the pure investment problem, they found the HJB solution in explicit form. 
Unfortunately, we can not apply this method to the general investment/consumption problem in view of an additional strongly nonlinear term due to the consumption. 
Moreover, even in the pure investment problem, (see, for example, \cite{BoguslavskyBoguslavskaya2004}) the HJB solution is extremely explosive, i.e.  
it goes to infinity in a squared exponential power ($e^{s^2}$) rate  as the variable corresponding to the risky assets  ($s$) 
goes to infinity in this financial market.
By this reason,  we can not use the analytical tool of Black--Scholes model to  proof  the verification theorem for spread market. 
 %
 %
%
In this paper we develop a new method for the probabilistic analysis of the parabolic PDE.
Similarly to \cite{BerdjanePergamenchtchikov2013}, 
we study the HJB equation  through the Feynman--Kac (FK) representation. To this end
we introduce a special metric space in which the FK mapping is contracted. Taking this into account we show 
the fixed-point theorem for this mapping and we show that the fixed-point solution is the classical solution for the HJB equation in our case. 
Moreover, by using the verification theorem  we find the optimal financial strategies.
 It turned out that the optimal investment and consumption strategies depend on the solution of a nonlinear parabolic partial differential equation. 
 Therefore, to calculate the optimal strategies one needs to study  numerical schemes.

The rest of the paper is organized as follows. 
In \cref{sec:Mm}
we introduce the financial market.
In \cref{sec:MP} 
we define all necessary parameters. 
In \cref{sec:HJB} 
we write the HJB equation. 
In \cref{sec:Mr} 
we state the main results of the paper.
In \cref{sec:PFKm} 
we study the properties of the FK mapping.
In \cref{sec:FP} 
we study the properties of the fixed-point function.
The proofs of the main results are given in \cref{sec:Pr}. 
The corresponding verification theorem is stated in the Appendix  
 with some auxiliary results.



\bigskip
\section{Market model}\label{sec:Mm}

Let 
$ (\Omega, \cF_\zs{T}, (\cF_\zs{t})_\zs{0\le t\le T}, \P)$
be a standard filtered probability space with 
$ (\cF_\zs{t})_\zs{0\le t\le T}$ 
adapted Wiener processes 
$ (W_\zs{t})_\zs{0\le t\le T} $. 
Our financial market consists of one {\em riskless bond}  
$ (\check{S}_\zs{t})_\zs{0\le t\le T}$ 
and {\em risky spread stocks}
$ (S_\zs{t})_\zs{0\le t\le T}$ 
governed by the following equations:
\begin{equation}\label{sec:Mm.1}
\left\{
		\begin{array}{ll}
								\d \check{S}_\zs{t}& = r\check{S}_\zs{t} \d  t, 
			\hfill
								\check{S}_\zs{0} = 1 ,
			\\[2mm]
							 \d  S_\zs{t}& =   - \kappa S_\zs{t} \d  t+ \sigma \d  W_\zs{t},
			\qquad
								S_\zs{0}>0 .
		\end{array}
\right.
\end{equation}
Here the constant
$ \kappa> 0$  
is the market mean-reverting parameter from $ \bbr$ 
and 
$ \sigma>0$ 
is the market volatility.
We assume that the interest rate
 $ r \le \kappa $.
Let now 
$ \check{\alpha}_\zs{t}$ be the number of riskless assets $ \check{S}$ 
and 
$ \alpha_\zs{t}$ be the number of risky assets in the moment $ 0 \le t \le T $,
 and the consumption rate is given by a non negative  integrated function $ (c_\zs{t})_\zs{0\le t\le T}$ \cite{Karatzasshreve1998}.
Thus the wealth process is 
$$
X_\zs{t} = \check{\alpha}_\zs{t} \check{S}_\zs{t}+ \alpha_\zs{t} S_\zs{t}.
$$
Using the self financial principle from 
\cite{Karatzasshreve1998}
we get
\begin{equation}\label{sec:Mm.1+1}
 \d  X_\zs{t} =   \check{\alpha}_\zs{t} \d \check{S}_\zs{t}+ \alpha_\zs{t}  \d  S_\zs{t} -  c_\zs{t}  \d  t.
\end{equation}
We define the financial strategy as
 $$
 \upsilon  =  (\upsilon_\zs{t})_\zs{0  \leq  t  \leq  T}  =  (\alpha_\zs{t}, c_\zs{t})_{0  \leq  t  \leq  T},
 $$
where 
$ c_\zs{t}\ge 0$
 is the consumption rate  
and 
$ \alpha_\zs{t}$ 
 is the investment position in the risky asset.
So, replacing now in \cref{sec:Mm.1+1} the differentials  $  \d  \check{S}_\zs{t}$  and 
$  \d  S_\zs{t}$  by their definitions in \cref{sec:Mm.1} we obtain the differential equation for the wealth process
\begin{equation}\label{sec:Md.3}
 \d  X_\zs{t}^{\upsilon}   =  (rX_\zs{t}^{\upsilon} -  \kappa_\zs{1} \alpha_\zs{t} S_\zs{t} ) \d  t + \alpha_\zs{t} \sigma  \d  W_\zs{t}  -  c_\zs{t}
 \d  t,
\end{equation}
where $ \kappa_\zs{1} = \kappa+r>0$.

\begin{definition} \label{De.sec:Md.3}
The financial strategy $ \upsilon = (\upsilon_\zs{t})_\zs{0  \leq  t  \leq  T}$
is called admissible if this process is adapted and
the equation \cref{sec:Md.3} has a unique strong nonnegative solution.
\end{definition}

\noindent 
We denote by $ \cV$ the set of all admissible financial strategies.
For initial endowment $ x>0$, 
admissible strategy $ \upsilon$ in $ \mathcal{V}$
and
the state process $ \varsigma_\zs{t} = (X^{\upsilon}_\zs{t},S_\zs{t})$ , we introduce for $ 0<\gamma<1$ 
the following value function 
\begin{equation}\label{sec:Md.4}
J(\varsigma, t, \upsilon) : =  \E_\zs{\varsigma, t} \bigg( \int_\zs{t}^T c^{\gamma}_\zs{u}  \d  u + \varpi (X_\zs{T}^\upsilon)^{\gamma}\bigg),
\end{equation}
where $ \varpi > 0$ is some fixed constant, $ \E_\zs{\varsigma, t}$ is the conditional expectation with respect to $ \varsigma_\zs{t}= \varsigma = (x,s)$. We set $J(\varsigma, \upsilon)= J(\varsigma, 0, \upsilon)$.
\\
Our goal is to maximize the value function \cref{sec:Md.4}, i.e.
\begin{equation}\label{eq: supvalue}
\sup_\zs{ \upsilon \in \cV} J(\varsigma, \upsilon )
\end{equation}
To do this we use the dynamical programming method. Therefore, we need to study the problem \cref{sec:Md.4} for any 
 $ 0 \le t \le T$.
 \begin{remark}
 The coefficient $0 < \varpi < \infty$, explains the investor preference between consumption and pure investment problem. Therefore, we did not consider the case where $\varpi = 0$, as in reality the trader is more interested in the terminal wealth than consumption. 
 \end{remark}
\bigskip

\section{Main Parameters}\label{sec:MP}

First we introduce the following ordinary differential equation
\begin{equation}\label{eq_g}
g'(t) - 2 \gamma_\zs{2} g(t)+ \gamma_\zs{1} g^2(t) + \gamma_\zs{3} = 0\quad\mbox{and}\quad g(T) = 0,
\end{equation}
where 
$$
  \displaystyle\gamma_\zs{1} = \frac{\sigma^2}{1 -  \gamma},\quad
\gamma_\zs{2} = \frac{\gamma \kappa_\zs{1}}{1 -  \gamma}+ \kappa ,
\quad
\kappa_\zs{1} =   \kappa+r
\quad
\mbox{and}
\quad
 \gamma_\zs{3} = \frac{\gamma \kappa_\zs{1}^2}{(1 -  \gamma) \sigma^2}  .
 $$
 One can check directly that 
\begin{equation}\label{funct_g}
g(t) =  \check{\gamma}_\zs{2} -  \vartheta -  \frac{2 \vartheta(\check{\gamma}_\zs{2} -  \vartheta)}{ \me^{\omega (T - t) }(\check{\gamma}_\zs{2}+\vartheta)  -  \check{\gamma}_\zs{2}+ \vartheta }  ,
\end{equation}
where 
$ \dsl\check{\gamma}_\zs{2} =  \gamma_\zs{2}/\gamma_\zs{1}, 
\quad
 \check{\gamma}_\zs{3} = \gamma_\zs{3}/\gamma_\zs{1},
 \quad
\omega =   2 \vartheta \gamma_\zs{1}
\quad
\mbox{and}
\quad 
\vartheta =  \sqrt{\check{\gamma}^2_\zs{2} -  \check{\gamma}_\zs{3}} $ .
\\[3mm]
Note that $g(t)$ is decreasing, i.e., $\max_\zs{0 \le t \le T} g(t)= g(0)$.
Taking into account that $r \le \kappa$,
we get 
$\check{\gamma}^2_\zs{2} \ge \check{\gamma}_\zs{3}$. Furthermore, we set
\begin{multline}\label{eq:B0.B1}
\B_\zs{1} = \frac{1}{\gamma_\zs{1}}
\left( 
	\sqrt{\frac{\pi}{2T}}
	+
	\sqrt{\frac{\vert\pi - 4T \sigma^2 \varpi^{\frac{1}{\gamma -1}}\vert}{2T}}
\right)
\quad
\mbox{and}
\quad
\B_\zs{0} = 
\left( 
		q_\zs{1} +\frac{\gamma_\zs{1}}{2} \bold{B}^2_\zs{1} 
\right)T 
 ,
\end{multline}
where
$ \dsl{q_\zs{1} = g(0)  \sigma^2 / 2 + r \gamma +  (1 -  \gamma)  \varpi^{\frac{1}{\gamma  - 1}}}$. 
We denote by $ C^{1,0}_\zs{+} \big(\bbr \times [0,T] \big)$
the set of all positive functions from $ C^{1,0} \big(\bbr \times [0,T] \big)$, i.e. the set of all 
$ \bbr \times [0,T] \to \bbr_\zs{+}$ continuous partial derivatives with respect to the first variable $s$ and continuous functions in the second variable $t$. 
Now
we introduce the following set
\begin{equation}\label{sec:FKm.1}
\cX = \Big\{ 
						h \in C^{1,0}_\zs{+} \big( \bbr \times [0,T]\big)
		: \quad
						 \sup_\zs{s, t} h(s, t)  \leq  \bold{B_\zs{0}}
		,\quad
						 \sup_\zs{s, t} |h_\zs{s}(s, t)| \leq  \bold{B_\zs{1}} \Big\}.
\end{equation}
For some $ \varkappa >1$,  which we will precise later, we introduce the metric in this space
\begin{equation}\label{sec:FKm.2}
		\rho(f,h) = \sup_\zs{\substack{s \in \mathbb{R} , 
																0  \leq  t  \leq  T}}
								 \me^{- \varkappa (T - t)}\Upsilon_\zs{f,h}(s,t)
,
\end{equation}
where $\Upsilon_\zs{f,h}(s,t) = |h(s, t) -  f(s, t)|+ |h_\zs{s}(s, t) -  f_\zs{s}(s,t)|$.
Now for any $ 0\le t\le T$ and $ s\in\bbr$ we introduce the the process $ (\eta^{s, t}_\zs{u})_\zs{t\le u\le T}$
as the solution of the following stochastic differential equation
\begin{equation}\label{sec:FKm.4}
\d  \eta^{s, t}_\zs{u}
 = g_\zs{1}(u) \eta^{s, t}_\zs{u}  \d  u + \sigma \d  \check{W}_\zs{u},
\qquad
 \eta^{s, t}_\zs{t} = s,
\end{equation}
where 
$ g_\zs{1}(t) =  \gamma_\zs{1}  g(t) -  \gamma_\zs{2}  $
and
$ \dsl{(\check{W}_\zs{u})_\zs{u\ge 0}}$ is a standard Brownian motion. It is clear that
$
 \eta^{s, t}_\zs{u}
 \sim  \mathcal{N} (s \, \mu(u,t), \sigma_\zs{1}^2(u,t)) ,
$
with
\begin{align}\label{eqn:mu&sigma}
\mu(u,t) =  \exp\bigg\{ \dsl\int^u_\zs{t} g_\zs{1}(\nu)  \d  \nu \bigg\}
				\quad
				\mbox{and}
				\quad
 \sigma_\zs{1}^2(u,t) = \sigma^2 \int^u_\zs{t} \mu^2(u,z)  \d  z .
 \end{align}
Now  for any 
$ h\in\cX$  
we define the FK  mapping  as
\begin{equation}\label{sec:FKm.3}
\cL_\zs{h}(s,t) =   \int^T_\zs{t}\E  \Psi_\zs{h}(\eta^{s, t}_\zs{u}, u)  \d u,
\end{equation}
where 
$ \dsl \Psi_\zs{h}(s, t)= \Gamma_\zs{0}\Big(s, t, h( s, t), h_\zs{s}( s, t)\Big)$ 
and
\begin{equation} \label{eqn: GAMMA0}
\Gamma_\zs{0}(s, t ,y_\zs{1},y_\zs{2})  =
\frac{\sigma^2 y^2_\zs{2}}{2(1 -  \gamma)} + \frac{\sigma^2 g(t)}{2} + r \gamma +(1 -  \gamma) \varpi_\zs{1}G(s, t,y_\zs{1}).
\end{equation}
Here, the coefficient 
$ \varpi_\zs{1}= \varpi^{- 1/(1- \gamma)}$
 and
\begin{equation}\label{sec:FKm.66}
G(s,t, y) = \exp \bigg\{ - \frac{1}{1 -  \gamma} \bigg(\frac{s^2}{2} g(t) +y \bigg)\bigg\}.
\end{equation}
In this paper we assume that $T < T_\zs{0}$ and
\begin{equation}\label{eq: T0}
T_\zs{0}=
 \min 
 		\bigg(\frac{\kappa (1- \gamma)}{2 (3+ \gamma) \kappa_\zs{2}} 
													  					, \frac{\gamma (1- \gamma)}{ (3+ \gamma)(\gamma+1) \sigma^2 g(0)} ,
 		\frac{\pi}{4 \sigma^2}
 		\bigg) ,
\end{equation}
where $\kappa_\zs{2}= \kappa^2_\zs{1} \Big(1/\sigma^2+ 1/2+  g(0)/\kappa_\zs{1} \Big). $

\begin{remark}\label{sec:FKm.1+++1}
Note that we use the FK  mapping
to study the HJB equation which will be defined in the next section.
\end{remark}
%
%
\section{Hamilton--Jacobi--Bellman equation }\label{sec:HJB}

Denoting by  $ \varsigma_\zs{t} = (X_\zs{t},S_\zs{t})$,
we can rewrite  equations \cref{sec:Mm.1} and \cref{sec:Md.3}
as,
\begin{equation}\label{sec:HJB.1}
 \d  \varsigma_\zs{t} = a(\varsigma_\zs{t}, \upsilon_\zs{t})  \d  t +b (\varsigma_\zs{t}, \upsilon_\zs{t})  \d  W_\zs{t},
\end{equation} 
where
$$
a(\varsigma, \u)  =  \begin{pmatrix} \ r x - \kappa_\zs{1} \alpha s  - c \\  -  \kappa s \end{pmatrix},
\quad 
b(\varsigma, \u) =  \begin{pmatrix} \ \alpha \sigma \\ \sigma \end{pmatrix}
\quad
\mbox{and}
\quad
\u = (\alpha,c)
.
$$
We introduce the Hamilton function, for any 
 $$
 q = \begin{pmatrix}
q_\zs{1} \\ q_2 \end{pmatrix}
  ,\quad
  M   =  \begin{pmatrix}
 M_{11} \ \ M_{12} \\ M_{21} \ \ M_{22} 
 \end{pmatrix} 
\quad
\mbox{and} 
\quad
 0 \leq  t  \leq  T,
 $$
 we set,
\begin{equation} \label{sec:HJB.2}
 H(\varsigma, t, q, M): =  \sup_{\u \in \Theta} H_0(\varsigma, t, q, M, \u), \qquad \Theta \in \mathbb{R} \times \mathbb{R}_+,
\end{equation}
 where
 $
  H_0( \varsigma, t, q, M, \u): =  a'(\varsigma, t, \u) q+ \frac{1}{2} \textbf{tr}[bb'(\varsigma, t, \u) M] +c^{\gamma},
$
with the prime $ '$ here denotes the transposition.
 \\
In order to find the solution to the value function \cref{eq: supvalue} we need to solve the HJB equation  which is given by
\begin{align} \label{Hamilton--Jacobi--Bellman equation}
\begin{cases}
 z_\zs{t} (\varsigma, t)+H(\varsigma, t, \partial z(\varsigma, t), \partial^2 z(\varsigma, t)) = 0 , 
 \qquad
  t \in [0,T] ,
  \\[2mm]
 z(\varsigma, T) = \varpi x^{\gamma}, 
 \hfill 
 \varsigma \in \mathbb{R}^2 ,
\end{cases}
\end{align}
where 
$$
\partial z(\varsigma, t) = \begin{pmatrix}
z_\zs{x} \\ z_\zs{s} \end{pmatrix} 
\qquad
 \mbox{and}
  \qquad
 \partial^2 z(\varsigma, t) = \begin{pmatrix}
z_{xx} \quad z_{xs} \\ z_{sx} \quad z_{ss} \end{pmatrix}.$$
Moreover, here
 \begin{align*}
H_0(\varsigma, t, q, M, \u) =  & \frac{\alpha^2 \sigma^2}{2} M_{11} +(  \sigma^2 M_{12} -  \kappa_\zs{1} s q_\zs{1}) \alpha + \frac{1}{2} \sigma^2 M_{22} 
\\[2mm]
&  + r x q_\zs{1} -  \kappa s q_2  -  c q_\zs{1} +c^\gamma.
 \end{align*}
From \cref{sec:HJB.2} we find that for $ q_\zs{1}>0$
 \begin{align} \label{sec:HJB.3}
 \alpha_\zs{0}(s,q,M)  = \frac{\kappa_\zs{1} s q_\zs{1}}{\sigma^2 M_{11}} -  \frac{M_{21}}{M_{11}} \quad \mbox{and} \quad
 c_\zs{0}(s,q,M) = (\frac{q_\zs{1}}{\gamma})^{\frac{1}{\gamma - 1}}.
 \end{align}
 By substituting these conditions into the HJB equation \cref{Hamilton--Jacobi--Bellman equation} for the value function, we obtain the following nonlinear PDE.  
 \begin{align} \label{sec:HJB.4}
 \begin{split}
  z_\zs{t}(\varsigma, t) +\frac{1}{2} \frac{( \sigma^2 z_{xs}  -  \kappa_\zs{1}s z_\zs{x})^2}{\sigma^2 |z_\zs{xx}|}+ \frac{\sigma^2 z_\zs{ss}}{2} +rx z_\zs{x}  - \kappa s z_\zs{s} \\
 + (1 - \gamma) (\frac{z_\zs{x}}{\gamma})^{\frac{\gamma}{\gamma - 1}} = 0 ,  
 \end{split}
 \end{align}
 where $ z(\varsigma, T) = \varpi x^{\gamma}$. 
\\
 To study this equation we use the following form for the solution
 \begin{equation}\label{HBJ_form++}
 z( x, s,t) =  \varpi x^\gamma
 U(s,t)
 \quad\mbox{and}\quad
 U(s,t) = \exp\bigg\{\frac{s^2}{2} g(t)+ Y(s,t) \bigg\}.
 \end{equation}
 The function $g(.)$ is defined in \cref{funct_g}, and 
 \begin{align} \label{sec:HJB.0}
\begin{cases}
Y_\zs{t}(s,t)+ \frac{1}{2}\sigma^2 Y_{ss}(s,t)+ s g_\zs{1}(t)Y_\zs{s}(s,t)+\Psi_\zs{Y}(s,t) = 0,
\\[2mm]
Y(s,T) = 0,
\end{cases}
\end{align} 
where $ \Psi_\zs{Y}(s,t)$ is given in \cref{eqn: GAMMA0}
and the function $ g_\zs{1}(.)$ is defined in
 \cref{sec:FKm.4}.
 \\
As we will see later, that the equation \cref{sec:HJB.0} has a solution in $ \C^{2,0}(\bbr \times [0,T])$
which can be represented as a fixed point for the FK mapping 
\begin{equation}  \label{sec:fixed_point_1}
h(s,t) =  \E \int^T_\zs{t} \Psi_\zs{h}( \eta^{s, t}_\zs{u}, u)  \d  u = \cL_\zs{h}(s,t)
\,.
\end{equation}
Using the solution by equation \cref{HBJ_form++}, we define the functions 
 \begin{align}\label{sec:alpha_c_00}
 \begin{split}
\check{\alpha}_\zs{0} \big( \varsigma, t \big)
		& = 
 		\frac{\kappa_\zs{1}  s   z_\zs{x}(\varsigma, t) }{\sigma^2  z_\zs{xx}(\varsigma, t)} -  \frac{ z_{xs}(\varsigma, t)}{ z_\zs{xx}(\varsigma, t)}
 		 =  
 		\check{\beta}(s, t)  x,
\\[3mm]  
\check{c}_\zs{0}(\varsigma, t)
		&
 = 
\bigg(\frac{ z_\zs{x}(\varsigma, t)}{\gamma} \bigg)^{\frac{1}{\gamma - 1}}
=  \check{G}(s,t)  x,
\end{split}
\end{align}
where 
$$
\check{\beta}(s,t)=\frac{1}{1- \gamma} \Big( s g(t)+ Y_\zs{s}(s,t)- \frac{\kappa_\zs{1}}{\sigma^2} s \Big)
\quad
\mbox{and}
\quad
\check{G}(s,t)= \varpi^{\frac{1}{\gamma -1}}  G(s, t, Y(s, t)).
$$
Now we set the following stochastic equation to define the optimal wealth process, i.e., we set 
\begin{equation}\label{eq:dXt*}
\d X^*_\zs{t}= a^*(t)  X^*_\zs{t}  \d t 
+
b^*(t)  X^*_\zs{t}  \d W_\zs{t} ,
\end{equation}
where 
$
\quad  a^*(t)= r- \kappa_\zs{1} S_\zs{t} \check{\beta}(S_\zs{t}, t)-  \check{G}(S_\zs{t},t)
\qquad
\mbox{and}
\qquad
b^*(t)= \sigma  \check{\beta}(S_\zs{t}, t).
$
\\
By It\^o formula we can obtain that 
$$
X^*_\zs{t}= x  \exp \bigg\{ \int^t_\zs{0} a^*(u) \d u \bigg\}  \cE_\zs{0, t}(b^*) \,.
$$
Using the stochastic differential equation \cref{eq:dXt*} we define the optimal strategies:
 \begin{align} \label{opt. stgy.11}
\alpha_\zs{t}^*  
 =   
 \check{\alpha}_\zs{0}(\varsigma^*_\zs{t}, t) 
\quad 
\mbox{and} 
\quad
c_\zs{t}^* 
 =   
\check{c}_\zs{0}(\varsigma^*_\zs{t}, t)  ,
\end{align}
where $ \varsigma^*_\zs{t}= (X^*_\zs{t}, S_\zs{t})$ and $X^*_\zs{t}$ is defined in \cref{eq:dXt*}. 
 \\
\begin{remark}
 Note,  the main difference in the HJB  equation \cref{sec:HJB.4} from the one in \cite{BoguslavskyBoguslavskaya2004}  is the last nonlinear term as we see we can not use the solution method from \cite{BoguslavskyBoguslavskaya2004}. One can check that the solution for pure investment problem from \cite{BoguslavskyBoguslavskaya2004} can be obtained in \cref{opt. stgy.11} as $\varpi \rightarrow \infty$.
 \end{remark}

\section{Main results}\label{sec:Mr}

First we study the HJB equation.
\begin{theorem}\label{sec:solY}
Assume that $ 0 < T <T_\zs{0}$, with $ T_\zs{0}$ is given in \cref{eq: T0}, then equation 
\cref{Hamilton--Jacobi--Bellman equation}
 is the solution defined by 
 \cref{HBJ_form++}, 
 where $ Y$ is the unique solution of 
 \cref{sec:HJB.0}
  in $ \cX$ 
  and is the fixed point for the FK mapping, i.e., 
$
 Y=\cL_\zs{Y}.
$
\end{theorem}
\begin{theorem}\label{sec: optimal 3}
Assume that $0<T<T_\zs{0}$,
 Then the optimal value of 
 $ J(t, \varsigma, \upsilon)$
  is given by
 $$
 \max_\zs{\upsilon \in \cV} J(\varsigma, t, \upsilon) =  J(\varsigma, t, \upsilon^*) =  \varpi x^\gamma U(s,t) ,
 $$
 where the optimal control 
 $ \upsilon^* =  (\alpha^*, c^*)$
  for all 
  $ 0  \leq  t  \leq  T$
   is given in \cref{opt. stgy.11} 
   with the function $ Y$ defined in \cref{sec:fixed_point_1}. 
 The optimal wealth process 
 $ (X_\zs{t}^*)_\zs{0  \leq  t  \leq  T}$ 
  is the solution to \cref{eq:dXt*}.
 \end{theorem}
\noindent
Let us now define the approximation sequence 
$ (h_\zs{n})_\zs{n\geq1} $
for 
$ h$
 as 
 $ h_\zs{0} = 0 $,
 and for 
 $ n\ge 1 $,
as
\begin{equation}
\label{sec:FKm.11}
h_\zs{n} = \cL_\zs{h_\zs{n - 1}} \,.
\end{equation}
In the following theorems we show that the approximation sequence goes to the fixed function $ h$, i.e. $ h =  \cL_\zs{h}$.
 \begin{theorem}\label{sec: optimal 2}
For any
$ 0<\delta<1/2$,
 the approximation 
$$
\parallel h - h_\zs{n} \parallel  \leq  O(n^{ -  \delta n})\quad \mbox{as} \quad n \rightarrow \infty.
$$
 \end{theorem}
\begin{remark}
Note that the convergence rate is super geometrical.
\end{remark}
 \noindent
 Now we define the approximation. We set
 $$
 \alpha^*_\zs{n}(\varsigma, t)= \check{\beta}_\zs{n}(s, t) x
 \quad
 \mbox{and}
 \quad
 c^*_\zs{n}(\varsigma, t)= \check{G}_\zs{n}(s,t) x
 $$
where
$$
\check{\beta}_\zs{n}(s,t)= \frac{1}{1- \gamma}
						\bigg( 
						s  g(t) + \frac{\partial h_\zs{n}(s, t)}{\partial s}  - \frac{\kappa_\zs{1}}{\sigma}  s 
						\bigg)
		\quad
		\mbox{and}
		\quad
\check{G}_\zs{n}(s,t)= \varpi^{\frac{1}{\gamma-1}}  G(s, t, h_\zs{n}(s,t)).
$$
\begin{theorem} \label{sec: optimal 1}
For any  
$0< \delta <1/2$ 
$$
\sup_\zs{\substack{\varsigma \\[2mm] 0 \le t \le T}} \Big( \big| \alpha^*(\varsigma, t) -  \alpha^*_\zs{n}(\varsigma, t) \big| + \big| c^*(\varsigma, t) -  c^*_\zs{n}(\varsigma, t) \big| \Big)  \leq  O(n^{ -   \delta n}) ,
				\quad
				\mbox{as} 
				\quad
				n \rightarrow \infty.
				$$
 \end{theorem}
\begin{remark}
 As it is  seen from \cref{sec:solY} the approximation scheme for the HJB equation implies the approximation for the optimal strategy  with  super geometrical rate. 
 i.e. more rapid than any geometrical ones.
 \end{remark}


\section{Properties of the Feynman--Kac mapping}\label{sec:PFKm}

We need to study the properties of the mapping (\ref{sec:FKm.3}).

\begin{proposition} \label{Pr.sec:FKm.0}
The space
$ (\cX, \rho)$
 is the completed metrical space.
\end{proposition}

%
\begin{proposition}\label{Pr.sec:FKm.1}
Assume that 
$ T\le \pi/16   $.
 Then 
 $ \cL_\zs{h}\in\cX$
for any 
$ h\in\cX $,
 i.e. 
 $ \cL_\zs{h}:\cX\to\cX $.
\end{proposition}
\begin{proof}
The function 
$ \cL_\zs{h}(s,t)$
 is  given in \cref{sec:FKm.3} and can be written as
\begin{align*}
\cL_\zs{h}(s,t) =    \frac{\sigma^2}{2} \int^T_\zs{t} g(u)  \d  u + \frac{\sigma^2}{2(1 - \gamma)} \E \int^T_\zs{t} h^2_\zs{s} ( \eta^{s, t}_\zs{u}, u)  \d  u +r \gamma(T - t)
\\
+ (1 -  \gamma)   \varpi^{\frac{1}{\gamma  - 1}}  \E \int^T_\zs{t} G (\eta^{s, t}_\zs{u}, u, h(\eta^{s, t}_\zs{u}, u))  \d  u ,
\end{align*}
with $ G(s, t, y)$ is given in \cref{sec:FKm.66}.
\\
Therefore, 
\begin{align}\label{Prop rep 2}
|\cL_\zs{h}(s,t)|  \leq   \frac{\sigma^2}{2}  g(0) (T - t) + \frac{\sigma^2}{2(1 - \gamma)} \pmb{B}_\zs{1}^2 (T - t)+r \gamma(T - t) \nonumber
\\
+ (1 -  \gamma)   \varpi^{\frac{1}{\gamma  - 1}}   (T - t)\le  \B_\zs{0} , 
\end{align}
where  $ \B_\zs{0}$ and $\B_\zs{1}$ are given in \cref{eq:B0.B1}.
Then by taking the derivative with respect to 
$ s $,
 we get
\begin{multline*}
\frac{\partial}{\partial s}\cL_\zs{h}(s,t) = \frac{\sigma^2}{2(1 - \gamma)} \frac{\partial}{\partial s}\E \int^T_\zs{t} h^2_\zs{s} ( \eta^{s, t}_\zs{u}, u)  \d  u 
\\[2mm]
+ (1 -  \gamma)   \varpi^{\frac{1}{\gamma  - 1}} \frac{\partial}{\partial s}\E \int^T_\zs{t} G (\eta^{s, t}_\zs{u}, u, h(\eta^{s, t}_\zs{u}, u))  \d  u .
\end{multline*}
From \cref{Lem: Prob.sol} and as
 $ \dsl{\parallel  G \big(\eta^{s, t}_\zs{u} , u, h( \eta^{s, t}_\zs{u}, u)  \big) \parallel_\zs{t, \infty}  \leq 1} $, 
 we have
$$
\bigg|
		\frac{\partial}{\partial s}\cL_\zs{h}(s,t) 
\bigg|
				 \leq  
\frac{\sigma}{(1 - \gamma)} \sqrt{\frac{2 (T - t)}{\pi}} \pmb{B}_\zs{1}^2+ (1 -  \gamma)   \varpi^{\frac{1}{\gamma  - 1}}  \frac{2}{\sigma} \sqrt{\frac{2 (T - t)}{\pi}} .
$$
 Then by taking into account the definition of 
 $ \pmb{B}_\zs{1}$
  in \cref{sec:FKm.1} we obtain,
$$
			\bigg|
			\frac{\partial}{\partial s}\cL_\zs{h}(s,t)
			\bigg|
			\leq  
			\frac{\sigma}{(1 -  \gamma)}  \sqrt{\frac{2T}{\pi}} \pmb{B}_\zs{1}^2 
		    + (1 - \gamma)   \varpi^{\frac{1}{\gamma  - 1}}  \frac{2}{\sigma} \sqrt{\frac{2T}{\pi}} 
			 \leq 
			\pmb{B}_\zs{1} .
$$
So, we get that $\cL_\zs{h} \in \cX$.
Hence \cref{Pr.sec:FKm.1}.
\end{proof} 
\begin{proposition}\label{Pr.sec:FKm.2}
For all 
$ f \in \cX$, for all 
 $ s $, and
 $ 0 \le t \le T $,
$$
\frac{\partial}{\partial s} \cL_\zs{f}(s, t) =  \int^T_\zs{t} 
	\Bigg(
				\int_\zs{\bbr} \Gamma_\zs{0} \Big( z, t, f(z, u), f_\zs{s}(z, u)  \Big) \; \varrho(s, t, z, u)   \d  z
			   \Bigg)
			    \d  u ,
$$ 
where $\Gamma_\zs{0}$ is as in \cref{eqn: GAMMA0} and 
\begin{equation}\label{eq:varrho}
\varrho(s, t, z, u) =  \frac{\partial}{\partial s } \varphi(s, t, z, u)  =   K  \frac{\mu(u,t) }{ \sigma_\zs{1} (u,t)} \; \varphi(s, t, z, u) \,, 
\end{equation}
where 
\begin{equation}\label{eq:varphi}
\varphi(s, z, u)= \frac{ \me^{-\frac{K^2}{2}}}{\sqrt{2 \pi} \sigma_\zs{1}(u, t)  }
\quad
\mbox{ and}
\quad
 K(s, z, u) =  \frac{z -  s  \mu (u, t) }{ \sigma_\zs{1}(u, t) } \,.
\end{equation}
\end{proposition}

\begin{proposition}\label{Pr.sec:FKm.3}
The mapping 
$ \cL$
 is constraint in $ \cX $,  i.e. for any 
 $0 < \lambda <1$,
 there exists $ \varkappa \ge 1$
 in the metric \cref{sec:FKm.2}
 such that for any $ h$ 
 and 
 $ f \in \cX$ 
\begin{equation}\label{rho(Lh.Lf)}
\rho(\cL_\zs{h},\cL_\zs{f})\le
\lambda
\rho(h,f)
\,.
\end{equation}
\end{proposition}

\begin{proof}
Using the definition of  the mapping $\mathcal{L}_h$ in \cref{sec:FKm.3}
we obtain that for any $h$  and  $f$ from  $\cX$,
\begin{align*}
\mathcal{L}_h - \mathcal{L}_f  = & \frac{\sigma^2}{2(1 - \gamma)} \E \int^T_\zs{t} \bigg( h^2_\zs{s}(\eta^{s, t}_\zs{u}, u) -  f^2_\zs{s}( \eta^{s, t}_\zs{u}, u) \bigg)  \d  u
\\
&+(1 - \gamma)   \varpi^{\frac{1}{\gamma  - 1}}  \E \int^T_\zs{t} \Bigg( G \bigg( \eta^{s, t}_\zs{u}, u, h( \eta^{s, t}_\zs{u},u)\bigg) -  G \bigg(\eta^{s, t}_\zs{u}, u, f( \eta^{s, t}_\zs{u},u) \bigg) \Bigg)  \d  u \,.
\end{align*}
Taking into account that 
 the function $G$ is lipschitzian, i.e. 
for any $y_\zs{1}\ge 0$ and $y_\zs{2}\ge 0$
$$ 
\big|G(s, t, y_\zs{1}) -  G(s,  t, y_2)\big|  \leq  \frac{1}{1 -  \gamma}| y_\zs{1} -  y_2|
\,,
$$
we obtain that
\begin{align} \label{Prob. sol. 2}
|\cL_\zs{h} -  \cL_\zs{f}|  \leq  \frac{\sigma^2}{2(1 - \gamma)} \Bigg|\int^T_\zs{t} \E & \Big(h^2_\zs{s}( \eta^{s, t}_\zs{u}, u) - f^2_\zs{s}( \eta^{s, t}_\zs{u}, u) \Big)  \d  u \Bigg| \nonumber
\\[2mm]
&+   \varpi^{\frac{1}{\gamma  - 1}}  \E  \int^T_\zs{t} \Big| h( \eta^{s, t}_\zs{s} , u) -  f( \eta^{s, t}_\zs{u}, u) \Big|  \d  u .
\end{align}
Recall that $f$ and $h$ belong to  $\cX$, i.e. the difference for the squares of their derivatives can be estimated as
$\big|h^2_\zs{s}(z, u) - f^2_\zs{s}( z, u)\big|  \leq  2 \bold{B_\zs{1}}  |h_\zs{s}( z, u) - f_\zs{s}( z, u)|$,
therefore, 
\begin{align*}
\big|\mathcal{L}_h(s,t) -  \mathcal{L}_f(s,t)\big|  \leq & \bigg(\frac{\sigma^2 \bold{B_\zs{1}}}{(1 -  \gamma)}+  \varpi^{\frac{1}{\gamma  - 1}}  \bigg) \int^T_\zs{t} \Upsilon^*_\zs{h,f}(u) \me^{ - \varkappa (T - u)} \me^{\varkappa (T - u)}  \d  u ,
\end{align*}
where 
$\dsl\Upsilon^*_\zs{h,f}(t) =  \sup_\zs{y \in \bbr^p} \Upsilon_\zs{h,f}(y, t)$. In view of the definition
 \cref{sec:FKm.1} 
\begin{align*}
\big|\mathcal{L}_h(s,t) -  \mathcal{L}_f(s,t)\big|   \leq & \bigg(\frac{\sigma^2 \bold{B_\zs{1}}}{(1 -  \gamma)}+  \varpi^{\frac{1}{\gamma  - 1}}  \bigg)  \rho(h,f) \int^T_\zs{t}  \me^{\varkappa (T - u)}  \d  u
 \\[2mm]
  \leq & \bigg(\frac{\sigma^2 \bold{B_\zs{1}}}{(1 -  \gamma)}+  \varpi^{\frac{1}{\gamma  - 1}}  \bigg) \frac{\rho(h,f)}{\varkappa} \me^{\varkappa (T - t)} .
\end{align*}
Therefore for all  $ 0  \leq  t  \leq  T $,
$$
\sup_{s \in \bbr} |\cL_h(s, t) - \mathcal{L}_f(s,t)|  \leq  \frac{\bold{\tilde{B}_\zs{1}}}{\varkappa} \rho(h,f) \me^{\varkappa (T-t)} ,
\quad
\bold{\tilde{B}_\zs{1}} =  \frac{\sigma^2 \bold{B_\zs{1}}}{(1 -  \gamma)}+\varpi^{\frac{1}{\gamma -1}} .
$$
The partial derivative of $ \mathcal{L}(s,t)$  with respect to $ s$  is given by
\begin{align*}
\frac{\partial}{\partial s} \mathcal{L}_h(s,t) = & \frac{\sigma^2}{2(1- \gamma)} \; \E \;\frac{\partial}{\partial s} \int^T_\zs{t}  h^2_\zs{s}(\eta^{s, t}_\zs{u}, u)  \d  u 
\\[2mm]
&+ (1- \gamma)  \varpi^{\frac{1}{\gamma -1}}  \E \; \frac{\partial}{\partial s} \int^T_\zs{t}  G(\eta^{s, t}_\zs{u}, u, h(\eta^{s, t}_\zs{u}))  \d  u .
\end{align*}
By taking the expectation we obtain
\begin{align*}
\frac{\partial}{\partial s} \mathcal{L}_h(s,t) = & \frac{\sigma^2}{2(1- \gamma)} \int^T_\zs{t} \int_{\bbr} h^2_\zs{s}(z, u) \frac{\partial}{\partial s} \varphi (z,u)  \d  z \;  \d  u
\\[2mm]
&+ (1- \gamma)  \varpi^{\frac{1}{\gamma -1}}   \int^T_\zs{t} \int_{\bbr} G(z,  u, h(z, u)) \frac{\partial}{\partial s}\varphi(z, u)  \d  z \;   \d  u ,
\end{align*}
where $\dsl\varrho(s,t,z,u)= \partial \varphi (s,t,z,u)/\partial s$ and $\varphi (s,t,z,u)$ is given in \cref{eq:varphi}.
Therefore,  for $u>t$ 
and for some constant $ \c^{*}>0$
\begin{equation}\label{eqn: int(rho)<c}
\sup_\zs{s \in \bbr}\int_{\bbr} |\varrho(s,t,z,u)|  \d  z \; \leq \; \frac{\c^{*}}{\sqrt{ (u-t)}} .
\end{equation}
Putting now
$\hat{\alpha}_\zs{1}= \sigma^2(2- 2\gamma)^{-1}$
and $\hat{\alpha}_\zs{2}= (1- \gamma) \varpi^{\frac{1}{\gamma -1}}$,
we obtain that
\begin{align*}
\big|\frac{\partial}{\partial s} \mathcal{L}_h(s,t)&-  \frac{\partial}{\partial s} \mathcal{L}_f(s,t) \Big| =   \bigg|\int^T_\zs{t} \int_{\bbr} \Big(\hat{\alpha}_\zs{1} (h^2_\zs{s}(z, u)-f^2_\zs{s}(z, u))
\\[2mm]
&+ \hat{\alpha}_\zs{2} ( G(z, u, h(z, u)) - G(z, u, f(z, u)) \Big) \varrho(s,t,z,u)  \d  z \d  u \bigg| .
\end{align*}
Note here,  that
$$
\Big| \hat{\alpha}_\zs{1} (h^2_\zs{s}(z, u)-f^2_\zs{s}(z, u))+\hat{\alpha}_\zs{2} ( G(z, u, h(z, u)) - G(z, u, f(z, u))  \Big| 
 \leq  
 \bold{B_\zs{2}} \Upsilon^*_\zs{f,h}(u) ,
$$
where $\bold{B_\zs{2}}=\Big(2\hat{\alpha}_\zs{1} \bold{\tilde{B}_\zs{1}} + \hat{\alpha}_\zs{2} (1- \gamma) \Big)$.
Thus,
$$
\bigg|\frac{\partial}{\partial s} \mathcal{L}_h(s,t)-  \frac{\partial}{\partial s} \mathcal{L}_f(s,t) \bigg| \leq   \bold{B_\zs{2}} \int^T_\zs{t} \Upsilon^*_\zs{f,h}(u) \bigg( \int_\zs{\bbr} |\varrho(s,t,z,u) |  \d  z \bigg) du .
$$
Using here the bound 
\cref{eqn: int(rho)<c}, we obtain that
$$
\Big|\frac{\partial}{\partial s} \mathcal{L}_h(s,t)-  \frac{\partial}{\partial s} \mathcal{L}_f(s,t) \Big| \; \leq \;  \bold{B_\zs{2}} \; \sqrt{\frac{2}{\pi}}  \int^T_\zs{t} \frac{1}{\sqrt{u-t}} \Upsilon^*_\zs{f,h}(u) \me^{- \varkappa(T-u)} \me^{ \varkappa(T-u)} \d u .
$$
Using again here the definition
\cref{sec:FKm.1} we get
\begin{multline*}
\big|\frac{\partial}{\partial s} \mathcal{L}_h(s,t)-  \frac{\partial}{\partial s} \mathcal{L}_f(s,t) \Big|  \leq   \sqrt{\frac{2}{\pi}}  \bold{B_\zs{2}}\; \rho(f,h) \int^T_\zs{t} \frac{\me^{ \varkappa(T-u)} }{\sqrt{u-t}}  \d  u
\\[2mm]
  \leq    \frac{2}{\pi} \bold{B_\zs{2}} \; \rho(f,h) \me^{ \varkappa(T-t)} \int^T_\zs{t} \frac{\me^{ -\varkappa(u-t)} }{\sqrt{u-t}}
 \d  u 
\le 
 \bold{B_\zs{2}}\;  \rho(f,h) \frac{\me^{ \varkappa(T-t)}}{\sqrt{\varkappa}} .
\end{multline*}
and,  therefore, 
$$
\big|\frac{\partial}{\partial s} \mathcal{L}_h(s,t)-  \frac{\partial}{\partial s} \mathcal{L}_f(s,t) \Big|  \leq   \bold{B_\zs{2}}  \rho(f,h) \frac{\me^{ \varkappa(T-t)}}{\sqrt{\varkappa}} .
$$
Thus
\begin{multline*}
\big| \cL_\zs{h}(s,t)-  \cL_\zs{f}(s,t) \big| + \bigg| \frac{\partial}{\partial s} \cL_\zs{h}(s,t)-  \frac{\partial}{\partial s}\cL_\zs{f}(s,t) \bigg| 
			\\
			 \leq   \Bigg( \frac{\bold{\tilde{B}_\zs{1}}}{\varkappa}   \me^{\varkappa (T-t)} +  \bold{B_\zs{2}} \; \frac{\me^{ \varkappa(T-t)}}{\sqrt{\varkappa}} \Bigg) \rho(f,h) \,.
\end{multline*}
So, taking into account that $ \varkappa >1$, we get
$$
\rho(\cL_{h}, \cL_{f})  \leq   \frac{\bold{\tilde{B}}_\zs{2}}{\sqrt{\varkappa}}  \rho(f,h) ,
\quad
\bold{\tilde{B}}_\zs{2}= \bold{\tilde{B}}_\zs{1}+ \bold{B_\zs{2}}.
$$
Choosing here $ \varkappa= (\bold{\tilde{B}}_\zs{2})^2/\lambda^2$
we obtain the inequality \cref{rho(Lh.Lf)}. 
Hence \cref{Pr.sec:FKm.3}.
\end{proof}

\begin{proposition}\label{Pr.sec:FKm.4}
For the mapping $  \cL $ there exists a unique fixed point $  h $   from $ \cX $,   i.e. $ \cL_\zs{h} = h $,   such that for any $ n\ge 1$   and  for any $ \varkappa > (\bold{\tilde{B}}_\zs{2})^2$
\begin{equation}
\label{sec:FKm.22}
\rho(h,h_\zs{n})\le \B^{*}\lambda^{n} ,
\quad 
\lambda= \frac{\bold{\tilde{B}}_\zs{2}}{\sqrt{\varkappa}},
\end{equation}
where
$
 \dsl\B^{*} =  (\B_\zs{0}+\B_\zs{1} )/(1- \lambda)$,
with $ \B_\zs{0}$ and $ \B_\zs{1}$ are defined in \cref{eq:B0.B1}.
\end{proposition}
\begin{proof}
We want to show that the approximation sequence $ (h_\zs{n})_\zs{n \geq 1}$ 
 converge to a fixed point $h $, where $h_0=0$  and  
$h_\zs{n}  = \cL_\zs{h_\zs{n-1}}$ for $n \geq 1$.
Using here \cref{Pr.sec:FKm.2}, we obtain that 
$\rho(h_\zs{n}, h_\zs{n+1}) =  \rho(\cL_\zs{h_\zs{n-1}}, \cL_\zs{h_\zs{n}})  \leq  \lambda  \rho(h_\zs{n-1}, h_\zs{n})$.
Therefore, 
$$
\rho(h_\zs{n}, h_\zs{n+1}) \leq  \lambda \rho(\cL_\zs{h_\zs{n-1}}, \cL_\zs{h_\zs{n}})  \leq  \lambda^2  \rho(h_\zs{n-2}, h_\zs{n-1})  \leq  ...  \leq  \lambda^n  \rho(h_\zs{0}, h_\zs{1}) .
$$
Note that \cref{sec:FKm.1} implies directly that $ \rho(h_\zs{0}, h_\zs{1}) \le \B_\zs{0}+\B_\zs{1}$.
So, for $ m>n $,
$$
\rho(h_\zs{n}, h_\zs{m}) \leq (\lambda^n+ \lambda^{n+1}+ ... + \lambda^{m-1}) (\B_\zs{0}+\B_\zs{1})
 \le 
 \sum_{i = n}^\infty \lambda^i 
 (\B_\zs{0}+\B_\zs{1})
 .
$$
Therefore, there exists $h$, such that $\rho (h_\zs{n}, h ) \rightarrow 0 $, i.e., for all $n$,
 we obtain \cref{sec:FKm.2}. Hence \cref{Pr.sec:FKm.4}.
\end{proof}
\section{Properties of the fixed-point function $h$}\label{sec:FP}

In this section we study some regularity properties for the function $ h$. First we study the smoothness with respect to the variable $ s $.
\begin{proposition} \label{prop: h/ s2-s1}
If $ h \in \cX$ is a fixed point for $ \cL$ i.e.  $ h = \cL_\zs{h} $, then for any 
$ 0  <  \beta  <  1 $,
$$
\sup_\zs{0  \leq  t  \leq  T} \sup_\zs{|s_\zs{1}| , |s_\zs{2}| }  \frac{ \big| h_\zs{s}(s_\zs{1}, t)- h_\zs{s}(s_\zs{2}, t) \big| }{|s_\zs{1}- s_\zs{2}|^\beta} < + \infty .
$$
\end{proposition}
\begin{proof}
$$
\frac{\partial}{\partial s} h(s, t) =\int^T_\zs{t}\int_\zs{\bbr} \Psi_\zs{h}( z, u) \; \varrho( s, t, z, u)\d  z\d  u ,
$$
where $ \Psi_\zs{h}( z, u)$ and $ \varrho( s, t, z, u)$ are given in \cref{eqn: GAMMA0} and \cref{eq:varrho} respectively.
\\
Therefore,
$$
\bigg| \frac{\partial}{\partial s} h( s_\zs{1}, t)- \frac{\partial}{\partial s} h(s_\zs{2}, t) \bigg| 
 = 
 \bigg| \int^T_\zs{t} \Psi_\zs{h}( z, u) \bigg(  \int_\zs{\bbr}  \Big( \varrho( s_\zs{1},t, z, u)  -  \varrho( s_\zs{2},t, z, u)  \Big)  \d  z \bigg)     \d  u \bigg| .
$$
If $  \delta > 1$ then,
\begin{align*}
\frac{1}{ \delta^\beta}  \Big| \frac{\partial}{\partial s} h( s_\zs{1}, t)- \frac{\partial}{\partial s} h(s_\zs{2}, t) \Big|   \leq  &
\int^T_\zs{t}  \bigg( \int_\zs{\bbr}  \big| \varrho( s_\zs{1},t, z, u) \big|  \d  z \bigg) \;  \d  u  \\[2mm]
&+ \; 
\int^T_\zs{t}  \; \bigg( \int_\zs{\bbr}  \big| \varrho( s_\zs{2},t, z, u) \big|  \d  z \bigg)\;  \d  u \;  <  + \infty .
\end{align*}
For $ 0 <  \delta  =  |s_\zs{1}- s_\zs{2} | <1$,
Then, 
$$
\frac{1}{ \delta^\beta}  \Big| \frac{\partial}{\partial s} h( s_\zs{1}, t)- \frac{\partial}{\partial s} h(s_\zs{2}, t) \Big|   \leq  
\pmb{B}_\zs{0} \; \;  \int^T_\zs{t}  \Bigg(  \int_\zs{\bbr}  \frac{ \big| \varrho( s_\zs{1},t, z, u) - \varrho( s_\zs{2},t, z, u) \big| }{ \delta^\beta}   \d  z \Bigg)   \d  u
$$
where from \cref{Prop rep 2},  $ \int^T_\zs{t} \Psi_\zs{h}( z, u) \;  \d  u  \leq  \pmb{B}_\zs{0}$, and $ \pmb{B}_\zs{0}$ is 
 given in \cref{eq:B0.B1}.
Let 
$$
I( \delta) =  \int^T_\zs{t}  \Bigg(  \int_\zs{\bbr}  \frac{ \big| \varrho( s_\zs{1},t, z, u) - \varrho( s_\zs{2},t, z, u) \big| }{ \delta^\beta}   \d  z \Bigg)   \d  u .
$$
Then we can rewrite it as 
\begin{align*}
I( \delta) =  \int^{t+ \delta_\zs{1}}_\zs{t}   \int_\zs{\bbr}  \frac{ \big| \varrho( s_\zs{1},t, z, u) - \varrho( s_\zs{2},t, z, u) \big| }{ \delta^\beta}   \d  z  \;   \d  u \; 
\\[2mm]
+ \int^{T}_\zs{t+ \delta_\zs{1}}   \int_\zs{\bbr}  \frac{ \big| \varrho( s_\zs{1},t, z, u) - \varrho( s_\zs{2},t, z, u) \big| }{ \delta^\beta}   \d  z  \;    \d  u .
\end{align*}
Putting $ \delta_\zs{1}= \delta^{2 \beta}$
we obtain that
\begin{align*}
I( \delta)
 \leq  
\frac{1}{ \delta^\beta} \int^{t+ \delta_\zs{1}}_\zs{t}   \Bigg( \int_\zs{\bbr}   \big| \varrho( s_\zs{1},t, z, u) \big|  \d  z+ \int_\zs{\bbr} \big| \varrho( s_\zs{2},t, z, u) \big|    \d  z \Bigg)    \d  u
\\[2mm]
+ \frac{1}{ \delta^\beta} \int^{T}_\zs{t+  \delta_\zs{1}}   \int_\zs{\bbr}  \big| \varrho( s_\zs{1},t, z, u) - \varrho( s_\zs{2},t, z, u) \big|     \d  z  \d  u.
\end{align*}
Taking into account the bound \cref{Prob. sol. 2}, we estimate the integral $ I(\lambda)$ as
$$
I( \delta) \leq  \frac{ \sqrt{2}}{\sqrt{\pi} } 
\; + \; \frac{1}{ \delta^\beta} \int^{T}_\zs{t+ \delta_\zs{1}}   \int_\zs{\bbr}  \big| \varrho( s_\zs{1},t, z, u) - \varrho( s_\zs{2},t, z, u) \big|    \d z  \;  \d u .
$$
Then
$$
I( \delta) \leq  \frac{ \sqrt{2}}{\sqrt{\pi} } 
\; + \; \frac{1}{ \delta^\beta} \int^{T}_\zs{t +  \delta_\zs{1}}   \int^{s_\zs{2}}_\zs{s_\zs{1}} \; \Bigg( \int_\zs{\bbr}  \big|  \varrho_\zs{s} (v, t, z, u)  \big| \;   \d z \Bigg) \;  \d v \;  \d u ,
$$
where
$$
\varrho_\zs{s}  =   
								\frac{\partial}{\partial s} \varrho(v, t, z, u)
								= \frac{\mu^2}{\sqrt{2  \pi} \sigma_\zs{1}^3} \me^{\frac{-K^2}{2}} \Big( K^2 -1 \Big),
$$
where $ K= K(s, z, u)$ is given in \cref{eq:varphi}.
Thus 
$$
|\varrho_\zs{s}|   \leq  \frac{1}{\sqrt{2  \pi} \sigma_\zs{1}^3} \me^{\frac{-K^2}{2}} \Big( K^2 +1 \Big) 
$$
and
$$
\int_\zs{\bbr} |\varrho_\zs{s}(v, t, z, u) | \d z  =  
\frac{ 1}{ \sigma_\zs{1}^2}  \int_\zs{\bbr}    | K^2 +1 | \me^{\frac{-K^2}{2}}    \d K
 \le  \frac{\c^{*}}{\sigma_\zs{1}^2} .
$$
Taking into account that  
$
\sigma_\zs{1}^{-2}  \leq  \c^{*} (u-t)^{-1}
$
for some $ \c^{*}>0 $,
we get
$$
I( \delta) \leq  \frac{2 \sqrt{2}}{\sqrt{\pi } } 
\; 
+
 \;
  \c^{*}  \delta^{1-\beta} \int^{T}_\zs{t+  \delta_\zs{1}} \frac{1}{ u-t} \;   \d u
  								 \le
  								  \c^{*} + \delta^{1-\beta} (|\ln \delta_\zs{1}| + \vert \ln T\vert) .
$$
Hence \cref{prop: h/ s2-s1}.
\end{proof}
\\[2mm]Now we need to study the smoothness property with respect to $ t $.
We show now  that the function $ f$ and its derivatives are H\"oldarians.
\begin{proposition}
\label{Pro:h.hs(t2-t1)}
Let $ h = \cL_\zs{h}$, with $  h \in \cX $. 
Therefore, for all $ t$, for all $ N \geq 1$, and $  0< \beta < 1/2 $,
$$
\sup_\zs{\substack{0  \leq  t_\zs{1}  \leq  T \\[2mm]  0  \leq  t_\zs{2}  \leq  T}}
\; \; 
\sup_\zs{|s|  \leq  N } \Bigg(  \frac{\big| h(s, t_\zs{1}) - h(s, t_\zs{2}) \big| + \big| h_\zs{s}(s, t_\zs{1})- h_\zs{s}(s, t_\zs{2}) \big|}{|t_\zs{1}- t_\zs{2}|^\beta} \Bigg) < + \infty \; .
$$
\end{proposition}

\begin{proof}
Firstly, note that
$$
 h(s,t) = \int^T_\zs{t} \Gamma (s,  t, u)  \d u
 \quad
 \mbox{and}
 \quad
  \Gamma (s, t, u) =   \int_\zs{\bbr} \; \Psi_\zs{h}( z, u)  \varphi(s, ,t , z, u) \;   \d z .
 $$
Therefore, for any $ 0 \le t_\zs{1} \le t_\zs{2} \le T$
$$
 h(s, t_\zs{2} )-  h(s, t_\zs{1} )  =  \int^T_\zs{t_\zs{2}} 
 						\Big( \Gamma (s, t_\zs{2}, u)  \d u
 						-  \Gamma (s, t_\zs{1}, u) \Big)  \d u
 - \int^{t_\zs{2}}_\zs{t_\zs{1}} \Gamma (s, t_\zs{1}, u)  \d u \; .
 $$
 Let now $ \delta= t_\zs{2}- t_\zs{1}$ and $\delta_\zs{1}= \delta^{2 \beta}$ for some $0 < \beta < 1/2$.
 Talking into account that $ \Gamma$   is bounded, we obtain that for some $ \c^{*}>0$ 
 \begin{equation}\label{eqn: h(s, t2)-h(s, t1)}
\frac{1}{ \delta^\beta} \big| h(s, t_\zs{2})-  h(s, t_\zs{1} ) \big|
 \le 
\frac{\c^{*}}{ \delta^{\beta}}   I( \delta) + c \;  \delta ^{1 - \beta} \; ,
 \end{equation}
 where
 $
 I(\delta)=   \int^{T}_\zs{t_\zs{2}} \int_\zs{\bbr}\vert \Omega(z, u)\vert \d z \d u$
 and $\Omega(z, u)=\varphi (s, t_\zs{2}, z, u) - \varphi (s, t_\zs{1}, z, u)$. 
 We reprent this term as $
 I(\delta)=  I_\zs{1}( \delta)+I_\zs{2}(\delta)$,
 where
 $$
 I_\zs{1}( \delta)=\int^{t_\zs{2}+ \delta_\zs{1}}_\zs{t_\zs{2}}   \int_\zs{\bbr}\vert\Omega(z, u)\vert \d z \d u
 \quad
 \mbox{and} 
 \quad
 I_\zs{2}(\delta)= \int^{T}_\zs{t_\zs{2} +  \delta_\zs{1}}\int_\zs{\bbr}\vert\Omega(z, u)\vert \d z \d u.
 $$
It is clear that  $I_\zs{1}( \delta)\le 2 \delta_\zs{1}$. To estimate the term $I_\zs{2}(\delta)$ note that
$$
\vert \Omega(z, u)\vert
=
|\varphi (s, t_\zs{2},z,u) - \varphi (s,t_\zs{1},z,u) | 
 \leq
\int^{t_\zs{2}}_\zs{t_\zs{1}} \big| \varphi_\zs{t} (s, \theta, z, u) \big|  \d \theta 
,
$$
where
$$
\varphi_\zs{t} (s, \theta, z, u)=
\frac{\partial}{\partial t}  \varphi (s, t, z, u) =  \Bigg( \frac{\sigma_\zs{2}(u, t)}{2 \; \sqrt{2 \pi} \; \sigma^3_\zs{1}(u, t)}\;  - \frac{K K'}{\sqrt{2 \pi}  \sigma_\zs{1}} \Bigg) 
\me^{-\frac{K^2}{2}} ,
$$
the prime $'$ is the derivative with respect to $t$ and $\sigma_\zs{2}=(\sigma_\zs{1})'$. Denoting by  $\mu_\zs{1}=\mu'$
we obtain that
\begin{align*}
K'  = \Big( \frac{z- s \mu}{\sigma_\zs{1}}  \Big)' \; 
 =  \; \frac{s \mu_\zs{1}}{\sigma_\zs{1}} - \frac{z- s \mu }{\sigma^2_\zs{1}} (\sigma_\zs{1})' 
  =    - \frac{s \mu_\zs{1}}{\sigma_\zs{1}} - \frac{1}{2} \; K \; \frac{\sigma_\zs{2}}{2 \sigma^2_\zs{1}} .
\end{align*}
Taking into account that,  
$\mu_\zs{1}$ is bounded, we obtain that  for some $ \c^{*}>0$
\begin{align*}
\bigg|\frac{\partial}{\partial t}  \varphi (s, t, z,  u) \bigg| 
& 
 \leq    \c^{*} \; (1+ |s|)  \me^{-\frac{K^2}{2}}   \frac{(K^2+ |K|+ 1)}{\sigma^3_\zs{1}} \; .
\end{align*}
Therefore, for some $ \c^{*}>0$ and $ u > t$
$$
\int_\zs{\bbr} \bigg|\frac{\partial}{\partial t}  \varphi (s, t, z,  u) \bigg|  \;  \d z 
 \leq   \frac{\c^{*}  (1 + |s|)}{u-t} ,
$$
and we get 
\begin{align*}
|I_\zs{2}( \delta)| 
& \leq    \c^{*} \;  (1+ |s|) \; \int^{t_\zs{2}}_\zs{t_\zs{1}} \; \Bigg( \int^{T}_\zs{t_\zs{2}+  \delta_\zs{1}}  \frac{1}{ u-\theta}   \d u \Bigg) \;  \d \theta
								\\[2mm]
								&  \le   
								 \c^{*} \;  (1+ |s|)  \;  \delta \; \int^{T}_\zs{t_\zs{2}+  \delta_\zs{1}} \frac{ \d u}{ u- t_\zs{2}} 
								\le  	\c^{*} \; (1+ |s|)  \delta  |\ln  \delta_\zs{1}| . 
\end{align*}
Therefore, for some $\c^{*}>0$
$$
\limsup_\zs{\delta\to 0}
\frac{1}{\delta^\beta} \big| h(s, t_\zs{2})-  h(s, t_\zs{1} ) \big|  
							\leq
							   \c^{*} \; (1+ |s|) .
$$
Now to prove the second part we firstly take the partial derivative of the function $ h$ which may represented by
\begin{equation}\label{eq: hs}
\frac{\partial}{\partial s} h(s, t) =  \frac{1}{\sqrt{2 \pi}} \int^T_\zs{t} \frac{\mu(u, t)}{\sigma^2_\zs{1}(u, t)} \Bigg( \int_\zs{\bbr} \Psi_\zs{h}(u, z) \;  K \; \me^{- \frac{K^2}{2}}\;   \d z \Bigg) \;   \d u .
\end{equation}
Then, 
\begin{align*}
\frac{\partial}{\partial s} h(s, t)
&
 =   \int^T_\zs{t} \frac{\mu(u, t)}{\sigma_\zs{1}(u, t)} \Bigg( \int_\zs{\bbr} \Psi_\zs{h}(s \mu + \sigma_\zs{1} \; K, u) K \; \frac{\me^{- \frac{K^2}{2}}}{\sqrt{2 \pi}}  \d K \Bigg)  \d u
\\[2mm]
&
 =  \int^T_\zs{t} \frac{\mu(u, t)}{\sigma_\zs{1}(u, t)} \Bigg(  \E \; \Psi_\zs{h}(s \mu(u, t) + \sigma_\zs{1}(u, t) \; \xi, u) \xi \Bigg)  \d u ,
\end{align*}
where  $ \xi  \thicksim \cN(0, 1)$.
%
So, we can represent the derivative \cref{eq: hs} as
$$
\frac{\partial}{\partial s} h(s, t) =   \int^T_\zs{t} \q(t,u)  \d u 
\quad
\mbox{and}
\quad
\q(t,u) =  \q_\zs{1}(t,u) \q_\zs{2}(t,u) ,
$$
where
$\q_\zs{1}(t,u) =  \E \xi \Psi_\zs{h}(s\mu(u, t) + \sigma_\zs{1}(u, t)\xi,u)$
and $\q_\zs{2}(t,u)= \mu(u, t)/\sigma_\zs{1}(u, t)$. 
Setting now
$\q_\zs{3}(u)  =   \q(t_\zs{2},u) - \; \q(t_\zs{1},u)$,
we obtain that
$$
\frac{\partial}{\partial s} h(s, t_\zs{2})- \frac{\partial}{\partial s} h(s, t_\zs{1})
=\int^T_\zs{t_\zs{2}}  \q_\zs{3}(u)   \d u - \int^{t_\zs{2}}_\zs{t_\zs{1}} \q(t_\zs{1},u)   \d u .
$$
Now we recall, that the function $\Psi_\zs{h}$ is bounded, i.e. $\vert \q(t,u)\vert \le \c^{*}/\sqrt{u-t} $ for some $\c^{*}>0$. 
Therefore,
\begin{align*}
\bigg| \frac{\partial}{\partial s} h(s, t_\zs{2})- \frac{\partial}{\partial s} h(s, t_\zs{1}) \bigg| 
						& \leq  
						\int^T_\zs{t_\zs{2}}  \big| \q_\zs{3}(u)   \big|  \d u
+\int^{t_\zs{2}}_\zs{t_\zs{1}} \frac{\c^{*}}{\sqrt{u-t_\zs{1}}}   \d u 
						\\[2mm]
						& \leq  I_\zs{1}^*( \delta) + I_\zs{2}^*(\delta) + 2 \c^{*} \sqrt{ \delta}  ,
\end{align*}
where
\begin{align*}
 I^*_\zs{1}( \delta) =  \int^{t_\zs{2}+  \delta_\zs{1}}_\zs{t_\zs{2}}  \big| \q_\zs{3}(u)  \big|  \d u
 \quad
 \mbox{and}
 \quad 
  I^*_\zs{2}( \delta) =  \int^{T}_\zs{t_\zs{2}+  \delta_\zs{1}}  \big| \q_\zs{3}(u)   \big|  \d u .
 \end{align*}
Similarly,  for $0 < t_\zs{1} < t_\zs{2}$
\begin{align*}
I^*_\zs{1}( \delta)   \leq   \c^{*} \int^{t_\zs{2}+  \delta_\zs{1}}_\zs{t_\zs{2}} \Bigg( \frac{1}{\sqrt{u- t_\zs{2}}} + \frac{1}{\sqrt{u- t_\zs{1}}} \Bigg)  \d u \quad
 					\le
                    4 \c^{*} \sqrt{ \delta_\zs{1}}.
\end{align*}
To estimate $I^*_\zs{2}( \delta) $, note that
\begin{align*}
\big| \q_\zs{3}(u)  \big| 
& =  \big| \q_\zs{1}(u, t_\zs{2}) \; \q_\zs{2}(u, t_\zs{2})   -   \q_\zs{1}(u, t_\zs{1}) \; \q_\zs{2}(u, t_\zs{1})  \big|  
\\[2mm]
& \leq  \Big|  \q_\zs{2}(u, t_\zs{2}) \;\Big(  \q_\zs{1}(u, t_\zs{2}) -  \q_\zs{1}(u, t_\zs{1}) \Big) \Big|
+
\Big| \q_\zs{1}(u, t_\zs{1})   \;\Big(  \q_\zs{2}(u, t_\zs{2})- \q_\zs{2}(u, t_\zs{1}) \Big) \Big|.
\end{align*}
Moreover, noting that
$$
\q_\zs{2}(u, t)  =  \frac{\mu(u, t)}{\sigma_\zs{1}(u,t)}  \leq  \frac{1}{\sigma_\zs{1}(u,t)}  \leq  \frac{\c^{*}}{\sqrt{u-t}} ,
 $$
we obtain that for $ u>t $,
$$
\big| \q_\zs{3}(u)   \big|   
					\leq
					c \bigg(  \big| \q_\zs{2}(u, t_\zs{2}) -  \q_\zs{2}(u, t_\zs{1}) \big| + \frac{1}{\sqrt{u-t_\zs{1}}} \big| \q_\zs{1}(u, t_\zs{2})- \q_\zs{1}(u, t_\zs{1})  \big| \bigg) .
$$
From the definition of $ q_\zs{1}$, we can obtain that for some $ \c^{*}>0$
$$
\big| \q_\zs{2}(u, t_\zs{2}) -  \q_\zs{2}(u, t_\zs{1}) \big| 
 \leq  \int^{t_\zs{2}}_\zs{t_\zs{1}}  \frac{1}{(u- \theta)^{\frac{3}{2}}}  \d \theta
 \leq  \frac{ \delta}{(u- t_\zs{2})^{\frac{3}{2}}} .
$$
and
$$
\big| \q_\zs{3}(u)    \big|   \leq  \frac{ \delta}{(u- t_\zs{2})^{\frac{3}{2}}}  
						\; + \; 
						 \frac{\c^{*}}{\sqrt{u-t_\zs{1}}} \big| \q_\zs{1}(u, t_\zs{2})- \q_\zs{1}(u, t_\zs{1})  \big| .
$$
It should be noted that \cref{prop: h/ s2-s1} implies that for any  $ 0 < \beta <1 $ and for some $\c^{*}>0$,
\begin{equation}\label{eq: Psih(s2)-Psih(s1)}
\Big| \Psi_\zs{h}(s_\zs{2}, t)- \Psi_\zs{h} (s_\zs{1}, t) \Big|  \leq  \c^{*}|s_\zs{2}-  s_\zs{1}|^\beta.
\end{equation}
So,
$$
\big| \q_\zs{1}(u, t_\zs{2})- \q_\zs{1}(u, t_\zs{1})  \big|
							 \leq 
							 \c^{*} \; (1+ |s|^\beta) \Big(  | \mu(u, t_\zs{2}) - \mu(u, t_\zs{1}) |^\beta +\big| \sigma_\zs{1}(u, t_\zs{2})- \sigma_\zs{1}(u, t_\zs{1}) \big|^\beta  \Big) .
$$
We recall that $|\sigma_\zs{1}(u, t_\zs{2})- \sigma_\zs{1}(u, t_\zs{1})| \le \delta/ \sqrt{u-t_\zs{2}} $.
Therefore, 
$$
 \big| \q_\zs{1}(u, t_\zs{2})- \q_\zs{1}(u, t_\zs{1})  \big| 
						  \leq 
						  (1+ |s|^\beta) \Big( \frac{ \delta^\beta }{(\sqrt{u- t_\zs{2}})^\beta} \Big)
						    \leq 
						    (1+ |s|^\beta) \Big( \frac{ \delta^\beta }{(\sqrt{u- t_\zs{2}})^\beta} \Big).
$$
Thus, 
$$
\big| \q_\zs{3}(u)  \big|   \leq  \frac{ \delta}{(u- t_\zs{2})^{\frac{3}{2}}} \; + \; \frac{(1+ |s|^\beta) }{\sqrt{u- t_\zs{2}}} \; \frac{ \delta^\beta}{(\sqrt{u- t_\zs{2}})^\beta}.
$$
Therefore,
\begin{align*}
I^*_\zs{2}( \delta) 
&
 \leq   (1+ |s|^\beta) \int^T_\zs{t_\zs{2}+  \delta_\zs{1}} \Bigg( \frac{ \delta}{(u- t_\zs{2})^{\frac{3}{2}}} \; + \; \frac{ \delta^\beta}{(u- t_\zs{2})^{\frac{\beta}{2}+ \frac{1}{2}} } \Bigg)  \d u
							\\[2mm]
							 &\leq   
							\Bigg( \frac{ \delta}{\sqrt{ \delta_\zs{1}}} +  \delta^{\beta} ( \delta_\zs{1})^{\frac{1- \beta}{2}}  \Bigg) \Big( 1+ |s|^\beta \Big) .
\end{align*}
Therefore, for any $ 0 < \beta < 1/2$
$$
\overline{\lim}_\zs{\delta \rightarrow 0} \frac{I^*_\zs{1}(\delta) + I^*_\zs{2}(\delta)}{\delta^\beta} < + \infty .
$$
Hence, \cref{Pro:h.hs(t2-t1)}.
\end{proof}

\bigskip

\section{Proofs}\label{sec:Pr}
\subsection{Proof of  \cref{sec:solY}}

 Let $ h \in \cX$  
be the fixed point for the mapping $\cL$, i.e. $h=\cL_\zs{h}$. 
 Consider now the following equation 
 \begin{equation}\label{eq: Y_t}
Y_\zs{t}(s,t) + \frac{\sigma^2  Y_\zs{ss}(s,t)}{2} + s g_1(t) Y_\zs{s}(s,t) + \Psi_\zs{h}(s, t)=0,
\quad
Y(s,T)=0 ,
\end{equation}
where $ g_\zs{1}(t) =  \gamma_\zs{1}  g(t) -  \gamma_\zs{2} >0 $ 
and $\Psi_\zs{h}(s, t)$ is given in \cref{sec:FKm.3}.
Then we change the variables as
$ u(s,t)= Y(s, T-t)$, so we get
\begin{equation}\label{eq: u_t}
u_\zs{t}(s,t) - \frac{\sigma^2  u_\zs{ss}(s,t)}{2} - s g_1(t) u_\zs{s}(s,t) - \Psi_\zs{h}(s, t)=0,
\quad
u(s,0)=0 .
\end{equation}
We can rewrite the previous equation as 
$$
u_\zs{t}(s,t) - \frac{\sigma^2  u_\zs{ss}(s,t)}{2} +a(s, t, u, u_\zs{s})=0,
\quad
u(s,0)=0 ,
$$
where $a(s, t, u, p)= - s g_1(t) p - \Psi_\zs{h}(s, t)$.
Taking into account that
$$
\Psi_\zs{max}=
\sup_\zs{s \in \bbr} \sup_\zs{t} \Psi_\zs{h}(s, t) 
<\infty
,
$$
we obtain that 
$
a(s, t, u, 0)u= - \Psi_\zs{h}(s, t)|u| \ge - \Psi_\zs{max} |u|$, 
i.e.  the condition in \cref{sec:CaPr.1} holds with
$\Phi(r)= \Psi_\zs{max}$ and
$b=0$.
In view of Propositions \ref{prop: h/ s2-s1} and \ref{Pro:h.hs(t2-t1)}, the function $ \Psi_\zs{h}$ satisfies the H\"older condition $ C_\zs{5})$ for any  
$ 0< \beta<1/2$.
By using \cref{CaPrt00 theorem} we obtain that \cref{eq: u_t} has a bounded solution. Therefore, there exists a solution of \cref{eq: Y_t}. 
In order to prove this proposition we use the probabilistic representation.
 Now, we define a stopping time $ \tau_\zs{n}$
  $$
 \tau_\zs{n} = \inf\big\{n\geq t: |\eta^{s,t}_\zs{u}| \geq n\big\} \wedge T , 
 $$
 where the process $(\eta^{s,t}_\zs{u})_\zs{u\ge t}$
 is defined in  \cref{sec:FKm.4}. 
By the It\^o formula we obtain that
\begin{align*}
Y(s,t)  = & -\int^{\tau_\zs{n}}_\zs{t}  \Big( (Y_\zs{t}( \eta^{s,t}_\zs{u},u) + g_1(u)  \eta^{s, t}_\zs{u}  Y_\zs{s}( \eta^{s,t}_\zs{u},u)  +\frac{\sigma^2}{2}  Y_\zs{ss}( \eta^{s,t}_\zs{u},u) \Big) \;  \d  u \\&-  \int^{\tau_\zs{n}}_\zs{t} Y_\zs{s}( \eta^{s,t}_\zs{u},u) \;  \d \check{W}_\zs{u} +Y( \eta^{s, t}_\zs{\tau_\zs{n}}, \tau_\zs{n}) .
\end{align*}
taking into account the equation \cref{sec:HJB.0}
we obtain that
$$
Y(s,t) =  \int^{\tau_\zs{n}}_\zs{t}  \Psi_\zs{h}( \eta^{s,t}_\zs{u}, u) \;  \d  u -  \int^{\tau_\zs{n}}_\zs{t} Y_\zs{s}( \eta^{s,t}_\zs{u},u)\;  \d  \check{W}_\zs{u} +Y(\tau_\zs{n}, \eta^{s, t}_\zs{\tau_\zs{n}})  .
$$
Taking into account that $\E \int^{\tau_\zs{n}}_\zs{t} Y_\zs{s}( \eta^{s,t}_\zs{u},u) \;  \d  \check{W}_\zs{u}=0$, we obtain
$$
Y(s,t) =  \E \int^{\tau_\zs{n}}_\zs{t}  \Psi_\zs{h}( \eta^{s,t}_\zs{u}, u) \;  \d  u
+\E Y(\tau_\zs{n}, \eta^{s, t}_\zs{\tau_\zs{n}}) .
$$
Note here that the solution of the equation \cref{eq: Y_t}
 is bounded. So,  by Dominated Convergence theorem and in view of the boundary condition in \cref{eq: Y_t}
we obtain that
\begin{align*}
 \lim_\zs{n \rightarrow \infty} \E Y( \eta^{s, t}_\zs{\tau_\zs{n}}, \tau_\zs{n})
 					= \E \lim_\zs{n \rightarrow \infty} Y( \eta^{s, t}_\zs{\tau_\zs{n}}, \tau_\zs{n})
 					= \E Y( \eta^{s, t}_\zs{T}, T)=0 .
\end{align*}
Moreover, taking into account that $ \Psi_\zs{h} \ge 0 $,
by the Monotone Convergence theorem we obtain
$$
Y(s,t)
 = \E \lim_\zs{n \rightarrow \infty} \int^{\tau_\zs{n}}_\zs{t}  \Psi_\zs{h}( \eta^{s,t}_\zs{u},u) \;  \d  u
							  =  
							 \E \int^{T}_\zs{t} \Psi_\zs{h}( \eta^{s,t}_\zs{u},u) \d u ,
$$
i.e. $Y(s,t)=\cL_\zs{h}(s,t)  = h$.
Hence \cref{sec:solY}. \fdem

\subsection{Proof of \cref{sec: optimal 3}}
To proof this theorem we use the verification theorem (\ref{VTh: A1}) and find the solution to the  HJB equation  using FK mapping with $h$ a fixed point for the mapping $\cL$.
Therefore, the function 
\begin{equation} \label{eq: z}
z(\varsigma, t) = \varpi x^\gamma U(s, t) 
\quad
\mbox{and}
\quad
U(s,t)= \exp{\Bigg\{ \frac{s^2}{2} g(t)+ h(s,t)\Bigg\}} ,
\end{equation}
is the solution of the HJB equation \cref{sec:HJB.4}.
By using this function we calculate the optimal control variables in \cref{sec:alpha_c_00} and we obtain the strategies \cref{eq:dXt*} - \cref{opt. stgy.11}.
Hence $\H_\zs{3}$). 
Now we want to check condition $\H_\zs{4}$). First note that the equation 
$$
 \d \varsigma^*_\zs{t} =  a^*( \varsigma^*_\zs{t}, t)  \d t + b^*( \varsigma^*_\zs{t}, t)  \d W_\zs{t}, \qquad t \geq 0, \quad \varsigma^*_0 = x
$$
is identical to the equation (\ref{sec:HJB.1}).
By the assumptions on the market parameters, all the coefficients of (\ref{sec:HJB.1}) are continuous and bounded. so the usual integrability and Lipschitz conditions are satisfied, this implies $\H_\zs{4}$).\\
\begin{lemma} \label{sec: verif 1}
There exists some  $ \check{\delta}>1$  such that for any  $ 0 \le t < T < T_\zs{0}$,  
$$
\sup_\zs{\tau \in \cM_\zs{t}}  \E  \Big( Z^{{\check{\delta}}}(\varsigma_\zs{\tau}, \tau )   |  \varsigma_\zs{t}= \varsigma \Big) < + \infty ,
$$
where $ \cM_\zs{t}$ is the set of all stopping times  with $ t \le \tau \le T$ and
the function $ z$ is given in \cref{eq: z}.
\end{lemma}
\noindent
\cref{sec: verif 1} yields condition $\H_\zs{5}$),
where
$
 z(\varsigma, t) = \varpi x^\gamma \exp \{ s^2 g(t) /2+ Y(s,t) \}.
$
\\
Now, the Verification Theorem \ref{VTh: A1} implies  \cref{sec: optimal 3}. \fdem
\subsection{Proof of \cref{sec: optimal 2}}
We set 
$
 \Delta_\zs{n}(y, t) =  h(y, t)- h_\zs{n}(y, t) .
$ 
So, 
\begin{align*}
\Upsilon^*_\zs{h,h_\zs{n}}(y, t) 
 = & \sup_{(y, t) \in \mathcal{K}} (| \Delta_\zs{n}(y, t)|+ |D_\zs{y}  \Delta_\zs{n}(y, t)|) 
				\le 
				 \me^{\varkappa T} \rho(h_, h_\zs{n}) 
				 \\
				&  \le 
				(\B_\zs{0}+\B_\zs{1}) \frac{\lambda^n}{1- \lambda} \me^{\varkappa T}
				=
				(\B_\zs{0}+\B_\zs{1}) \exp\{ H (\lambda) \},
\end{align*}
where $ \B_\zs{0}$ and $\B_\zs{1}$ are defined in \cref{eq:B0.B1} and 
$H(\lambda)= \tilde{\B}_\zs{2} T/ \lambda^2 + n \ln n - \ln (1- \lambda)$.
%
\\
If we take $\lambda= 1/\sqrt{n}$ and $\varkappa= n(\tilde{\B}_\zs{2})^2$ 
then we obtain
$$
\Upsilon^*_\zs{h,h_\zs{n}}(y, t)= O(n^{-\delta n}) ,
$$
for any  $0<\delta <1/2$. Hence \cref{sec: optimal 2}.\fdem
\bigskip

\renewcommand{\theequation}{A.\arabic{equation}}
\renewcommand{\thetheorem}{A.\arabic{theorem}}
\renewcommand{\thesubsection}{A.\arabic{subsection}}

\section{Appendix}\label{sec:A}
\setcounter{equation}{0}
\setcounter{theorem}{0}

\bigskip

\subsection{Verification theorem} \label{VTh: A1}

Now we give the verification theorem from \cite{BerdjanePergamenchtchikov2013}.
Consider on the interval $ [0,T] $. The stochastic control process given by the $ N$-dimensional It\^o process
\begin{align}\label{sec:Model.00}
\begin{cases}
d \varsigma^{\upsilon}_\zs{t} & =  a(\varsigma^{\upsilon}_\zs{t}, t,  \upsilon) dt+ b(t, \varsigma^{\upsilon}_\zs{t}, \upsilon) dW_\zs{t} ,
 \quad 
 t\geq 0 , 
\\
\varsigma^{\upsilon}_0 & =  x\in \mathbb{R}^N , 
\\
\end{cases}
\end{align}
where $ (W)_{0 \leq  t  \leq  T} $  is a standard $ k-$ dimensional Brownian motion. We assume that the control process $ \upsilon$ takes values in some set $ \Theta $. Moreover, we assume that the coefficients $ a$ and 
$ b$ satisfy the following conditions:
\begin{itemize}
\item[$V_\zs{1}$)] {\sl for all $ t \in [0,T] $  the functions $ a(., t, .)$  and $ b(., t, .,) $ are continuous on $ \mathbb{R}^N \times \Theta ;$} where $ \Theta \in \mathbb{R} \times \mathbb{R}_{+} $. 
\item[$V_\zs{2}$)] {\sl for every deterministic vector $ \upsilon \in \Theta $  the stochastic differential equation 
$$
d \varsigma^{\upsilon}_\zs{t}  =  a(\varsigma^{\upsilon}_\zs{t}, t,  \upsilon) dt+ b(\varsigma^{\upsilon}_\zs{t}, t, \upsilon) dW_\zs{t}
$$
has a unique strong solution.}
\end{itemize}

Now we introduce admissible control process for the equation (\ref{sec:Model.00}).
\\
We set 
$$
\mathcal{F}_\zs{t} =  \sigma \{ W_\zs{u}, 0  \leq  u  \leq  t\} ,
\quad
\mbox{for any}
\quad
0 <t  \leq  T ,
$$
where a stochastic control process 
$ \upsilon = (\upsilon_\zs{t})_{t\geq0} = (\alpha_\zs{t}, c_\zs{t})_{t \geq 0}$ 
 is called admissible on $ [0,T] $  
 with respect to equation (\ref{sec:Model.00})
  if it is $ (\mathcal{F}_\zs{t})_{0  \leq  t  \leq  T} $  progressively measurable
   with values in $ \Theta$, 
   and equation (\ref{sec:Model.00}) has a unique strong a.s. continuous solution
    $ (\varsigma_\zs{t}^{\upsilon})_{0  \leq  t  \leq  T} $  such that 
$$
\int_0^T (|a(\varsigma_\zs{t}^\upsilon, t, \upsilon_\zs{t})|+|b(\varsigma_\zs{t}^\upsilon, t, \upsilon_\zs{t})|^2 )dt < \infty \quad \mbox{a.s.} .
$$
We denote by $ \mathcal{V} $  the set of all admissible control processes with respect to the equation (\ref{sec:Model.00}).\\
Moreover, let $ \dsl\bold{f}:  \mathbb{R}^m \times  [0,T]  \times \Theta\rightarrow [0,\infty) $ and $ \bold{h}: \mathbb{R}^m \rightarrow [0, \infty) $  be continuous utility functions. We define the cost function by
$$
\pmb{J}(x, t, \upsilon) =  \E_{x, t} \Bigg( \int^T_\zs{t} \bold{f} (\varsigma, u, \upsilon_\zs{u}) \d u + \bold{h}(\varsigma^\upsilon_\zs{T} ) \Bigg), \quad 0 \leq  t  \leq  T,
$$
where $ \E_{x, t} $ is the expectation operator conditional on $ \varsigma^{\upsilon}_\zs{t} = x$. Our goal is to solve the optimization problem
\begin{align}\label{sec:Model.01}
\pmb{J}^*(x, t) : =  \sup_{\upsilon \in \mathcal{V}} \pmb{J}(x, t, \upsilon) .
\end{align}
To this end we introduce the Hamilton function, i.e. for any $\varsigma$ and $ 0  \leq  t  \leq  T$,  with $ \bold{q} \in \mathbb{R}^N $ and symmetric $ N \times N $  matrix $ \bold{M} $  we set
$$
H(\varsigma, t, q, M): =  \sup_{\theta \in \Theta} H_0(\varsigma, t, \bold{q}, \bold{M}, \theta ) ,
$$
where 
$$
 H_0(\varsigma, t, \bold{q}, \bold{M}, \theta ) : =  a'(\varsigma, t, \theta) \bold{q} + \frac{1}{2} tr [bb'(\varsigma, t, \theta) \bold{M}] + \bold{f}(\varsigma, t, \theta) .
 $$
In order to find the solution to (\ref{sec:Model.01}) we investigate the HJB equation 
\begin{align}
\begin{cases}
z_\zs{t}(\varsigma, t) + H(\varsigma, t, z_{\varsigma}(\varsigma, t), z_{\varsigma \varsigma} (\varsigma, t)) = 0, \quad t \in [0,T], \\
z(\varsigma, T) =  \bold{h}(\varsigma), \quad \varsigma \in \mathbb{R}^N .
\end{cases}
\end{align}
Here $ z_\zs{t} $  denote the partial derivatives of $ z $  with respect to $ t,  z_{\varsigma}(\varsigma, t) $ the gradient vector with respect to $ \varsigma $  in $ \mathbb{R}^N $  and $ z_{\varsigma \varsigma} (x, t) $  denotes the symmetric hessian matrix, that is the matrix of the second order partial derivatives with respect to $ \varsigma $.
\\
We assume the following conditions hold:
\begin{itemize}
\item[$\H_\zs{1}$)] {\sl The functions $ \bold{f} $  and $ \bold{h} $ are non negative.}
\item[$\H_\zs{2}$)] {\sl There exists a function $ z(\varsigma \; \mbox{from} \; \C^\zs{2,1} ( \mathbb{R}^N \times [0,T] ) , t) $  from $  \bbr^N \times [0,T]  \rightarrow (0, \infty) $  which satisfies the HJB equation .}
\item[$\H_\zs{3}$)] {\sl There exists a measurable function $ \theta^*: \mathbb{R}^N \times [0,T]  \rightarrow \Theta $   such that for all  $ \varsigma \in \mathbb{R}^N $ and $ 0  \leq  t  \leq  T $
$$
H(\varsigma, t, z_{\varsigma}(\varsigma, t), z_{\varsigma \varsigma} (\varsigma, t)) =  H_0 (\varsigma, t, z_{\varsigma}(\varsigma, t), z_{\varsigma \varsigma} (\varsigma, t), \theta^*(\varsigma^0, t)) .
$$}
\item[$\H_\zs{4}$)] {\sl There exists a unique strong solution to the It\^o equation
$$
d \varsigma^*_\zs{t}  =  a(\varsigma^*_\zs{t}, t) dt+ b( \varsigma^*, t) dW_\zs{t}, \quad \varsigma^*_0 = x, \quad t \geq 0  ,
$$
where $ a(., t) = a(., t,  \theta^*(., t)) $ and  $ b(., t) = b(., t, \theta^*(., t)) $. Moreover, the optimal control process $ \upsilon^*_\zs{t} =  \theta^*( \upsilon^*_\zs{t}, t) $ for $ 0  \leq  t  \leq  T $  belongs to $ \mathcal{V} $.}
\item[$\H_\zs{5}$)] {\sl 
$$
\E \left( \int_\zs{0}^T f(u) \d u + h(\varsigma^*_\zs{T})  \right) < + \infty \,.
$$ 
 }
\end{itemize}

\begin{theorem}
Assume that conditions $ \H_\zs{1}$)- $ \H_\zs{5}$) holds
$$
\Rightarrow \upsilon_\zs{t}^* = (\upsilon_\zs{t}^*)_{0  \leq  t  \leq  T}
$$
is a solution for this problem.
\end{theorem}

\bigskip
\subsection{Cauchy Problem}\label{Apdx:Chyplm}
Suppose $ u(x,t) $ is the classical solution of the following a nonlinear problem 
\begin{align}\label{sec:Main equation}
\begin{cases}
\cL u \equiv u_\zs{t}  -  \sum_\zs{1  \leq  i, j  \leq  n} a_\zs{ij}(x,t, u, u_x)u_\zs{x_\zs{i} x_\zs{j} }+ a(x, t, u, u_x) = 0 \\[3mm]
u|_\zs{t = 0} = u(x, 0) = \psi_0(x).
\end{cases}
\end{align}
We assume that there exist some functions $(a_\zs{1}, a_\zs{2}, ..., a_\zs{n}) ,$ such that
\begin{align}\label{sec:CaPr.00}
a_\zs{ij}(x, t, u, p)  &\equiv \frac{\partial a_\zs{i}(x, t,  u, p)}{\partial p_\zs{j}} 
\quad
\mbox{and}
\\
A(x, t, u, p)  &\equiv a(x, t, u, p)- \sum^{n}_\zs{i = 1} \frac{\partial a_\zs{i}}{\partial u} p_i- \sum^{n}_\zs{i = 1} \frac{\partial a_\zs{i}}{\partial x_\zs{i}} \; .
\end{align}
Now for any  $ N \geq 1 $,
$$
\Gamma_\zs{N} =  \{ (x, t):  \;  |x|  \leq  N, \; 0  \leq  t  \leq  T \} 
$$
We introduce the following conditions for ensuring the existence of the solution $ u(x, t)  $ of Cauchy problem.

Suppose that the following conditions hold.

\begin{itemize}
\item[$\C_\zs{1}$)] {\sl For all $ N \geq 1$,  
\begin{align*}
\psi_0(x) \in \cH^{2+\beta}(\Gamma_\zs{N})
\quad
\mbox{and}
\quad 
\max_\zs{E_\zs{n}} \mid \psi_\zs{0}(x) \mid <\infty .
\end{align*} }
\item[$\C_\zs{2}$)] { There exist $ h \geq 0 $ and some $ \bbr_\zs{+} \rightarrow \bbr_\zs{+}$ function $ \Phi $,
such that for all $ \dsl{x \in \bbr^n}$, $u \in \bbr$ and for all $ 0 \le t \le T$
\begin{equation}\label{sec:CaPr.1}
 A(x, t, u, 0) u \geq - \Phi (|u|) |u| - b ,
\end{equation} 
and
$$
\int_\zs{0}^\infty \frac{d \tau}{ \Phi (\tau)} =  \infty  .
$$

}
\item[$\C_\zs{3}$)] {\sl For $ t \in (0,T]$ and arbitrary $ x,\;  u,\;  p \in \bbr^n $, and any $ \xi =  (\xi_\zs{1}, \xi_\zs{2}, ..., \xi_\zs{n}) \in \bbr^{n}$, there exists $ 0 < \nu < \mu$  such that
\begin{align*}
\sum_\zs{1  \leq  i, j  \leq  n} a_\zs{ij} (x, t,  u, p)\xi_\zs{i} \xi_\zs{j} \geq 0 
\quad
\mbox{and}
\quad 
 \nu |\xi|^2  \leq  a_\zs{ij} (x, t,  u, p) \xi_\zs{i} \xi_\zs{j}  \leq  \mu |\xi|^2  .
\end{align*}
}
 \item[$\C_\zs{4}$)] 
{\sl The functions $ a_\zs{i} (x, t, u, p) $ and $ a(x, t, u, p)$ are continuous, 
and  the functions $ \dsl{(a_\zs{i})_\zs{1  \leq  i  \leq  n} }$ are differentiable with respect to $ x, u$ and $ p \in \bbr^n$ ,
and for any $ N \geq 1\;$  there exists $ \mu_\zs{1}  =  \mu_\zs{1} (N)$ such that
 $$
\sup_\zs{\substack{(x, t)  \leq  \Gamma_\zs{N} \\[1.5mm] |u|  \leq  N \\[1.5mm] p \in \bbr^n}} \frac{ \sum^{n}_\zs{i = 1} \big( |a_i| + \big| \frac{\partial a_i}{\partial u} \big| \big) (1+|p|) + \sum^n_{i, j = 1} |\frac{\partial a_i}{\partial x_\zs{j}}| +|a| }{(1+ |p|)^2} \leq  \mu_\zs{1}(N) . 
 $$
} 
 \item[$\C_\zs{5}$)] 
 {\sl  For all $ N \geq 1,$ 
 and for all $ |x|  \leq  N$,
  $ 0  \leq  t  \leq  T$,
   $ |u|  \leq  N$ 
   and $ |p|  \leq  N $,
    the functions $ a_i , \; a , \;  \partial a_i / \partial p_j, \; \partial a_i / \partial u$,
     and $ \partial a_i / \partial x_i$ are continuous functions satisfying a H$\ddot{o}$lder condition in $ x,\;  t,\;  u$ and $ p$ with exponents $ \beta,\;  \beta / 2,\;  \beta$ and $ \beta$ respectively for some $ \beta > 0$.
     
     }    
\end{itemize}

\begin{theorem}[ See Theorem 8.1, p. 495 of \cite{LadyzenskajasolonnikovUralceva1967}
]\label{CaPrt00 theorem}
Assume that the conditions $ \C_\zs{1}$)--$ \C_\zs{5}$) hold. Then there exists at least one solution $ u(x,t)$ of Cauchy problem (\ref{sec:Main equation}) that is bounded in $ R^N \times [0, T]$ which belongs to $ \cH^{2+\beta, 1+\beta/2}(
\Gamma_\zs{N})$ for any $ N \geq 1$.
\end{theorem}


\subsection{Proof of \cref{Pr.sec:FKm.0}}
One can check directly that the set $\cX$ is closed in the 
$\dsl{\C^{1,0}\left(\bbr\times [0,T]\right)}$ which is complete. So, the space $(\cX,\rho)$
is complete also. Hence \cref{Pr.sec:FKm.0}.\fdem

\subsection{Proof of \cref{Pr.sec:FKm.2}}
Firstly, note that from the definition of the mapping in  \cref{sec:FKm.3} we get for any $\delta >0$
$$
\frac{\cL_\zs{h}(s +  \delta, t)   -   \cL_\zs{h}(s,t)}{ \delta}  =  \int^T_\zs{t} \int_\zs{\bbr}  \Bigg( \Psi_\zs{h} (z,  u)   \bigg( \frac{\varphi (s+  \delta, t, z, u)   -   \varphi(s, t, z, u)}{ \delta} \bigg)  \d  z \Bigg) 
 \d  u .
$$
Taking into account that the function $ \rho$ is continuously differentiable, we can rewrite 
\begin{align*}
\frac{\varphi (s+  \delta, t, z, u)   -   \varphi(s, t, z, u)}{ \delta} & =  \frac{1}{ \delta} \int^{s+  \delta}_\zs{s} \varrho(\nu, t, z, u)  \d  \nu 
 =  \varrho(s, t, z, u) + \pmb{D}_\zs{ \delta} (s, t, z, u)  ,
\end{align*}
where 
$ \dsl\varrho(s, t, z, u) =  \partial  \varphi(s, t, z, u) / \partial s $  and
$$
\pmb{D}_\zs{ \delta} (s, t, z, u)  =  \frac{1}{ \delta}  \int^{s +  \delta}_\zs{s} \Big( \varrho (\nu, t, z, u)   -   \varrho (s, t, z, u)  \Big)  \d  \nu .
$$
So, 
$$
\frac{\cL_\zs{h}(s +  \delta, t)   -   \cL_\zs{h}(s,t)}{ \delta} 
 = 
 \int^T_\zs{t} \Bigg(  \int_\zs{\bbr}   \Psi_\zs{h} (z,  u)    \varrho (s, t, z, u)   \d  z\Bigg)   
 \d  u
+ \pmb{G}_\zs{ \delta} ,
$$
where
$
 \pmb{G}_\zs{ \delta}  =  \int^T_\zs{t} \Big(  \int_\zs{\bbr}   \Psi_\zs{h} (z,  u)    \pmb{D}_\zs{ \delta} (s, t, z, u)   \d  z \Big)   
 \d  u .
$
Now we have to prove that the term $\pmb{G}_\zs{ \delta}$ goes to zero as $\delta \rightarrow 0 $.
\\
As $ \Psi_\zs{h}(s,t)$ is a bounded function as for $h \in \cX $, therefore,
$$
|\pmb{G}_\zs{ \delta} |  
\leq  
\Psi^* \int^T_\zs{t} \frac{1}{ \delta}  \Bigg(  \int^{s +  \delta}_\zs{s}   \pmb{L}(\nu,  u)  \d  \nu   \Bigg)  \d u 
\le \Psi^*  \int^T_\zs{t} \pmb{L}^*_\zs{\delta}( u)  \d  u
$$
where 
$ \Psi^*=\sup_{\substack{z \in \bbr , \; 0\le t \le T}} |\Psi_\zs{h}(z,u)| , \;$ 
 $ \pmb{L}^*_\zs{\delta}( u)= \max_\zs{s \le \nu \le s+ \delta} \pmb{L}(\nu, u)$ 
 and
$$
\pmb{L}(\nu, u)  =   \int_\zs{\bbr}  | \varrho (\nu, t, z, u)   -   \varrho (s, t, z, u)  |  \d  z .
$$
We can check directly that for some $\c^{*}>0$
$$
\sup_{\substack{0 < \delta <1}} \pmb{L}^*_\zs{\delta}(u)   \leq  \frac{\c^{*}}{\sqrt{u -  t}} .
$$
Moreover, note that for some $N>1$
\begin{align*}
\pmb{L}(\nu, u)   \leq  \int\limits_\zs{|z|  \leq  N}  | \varrho (\nu, t, z, u)   -   \varrho (s, t, z, u)  |  \d  z
+
\int\limits_\zs{|z| > N}  | \varrho (\nu, t, z, u)   -   \varrho (s, t, z, u)  |  \d  z.
\end{align*}
 \\
The first part approach zero when $ N \rightarrow 0$.
\begin{align*}
\int\limits_\zs{|z| > N}   | \varrho (\nu, t, z, u)   |  \d  z  
 &=  
\frac{\mu(u,t)}{\sqrt{2  \pi}\sigma_\zs{1}(u,t)} \int\limits_\zs{|\sigma_\zs{1} y + \nu \mu| >N} |y| \me^{- \frac{y^2}{2}} \d y
\\[2mm]
\; &\leq \; 
\frac{\mu(u,t)}{\sqrt{2  \pi} \sigma_\zs{1}(u,t)}    
\int\limits_{|y|> N_\zs{1}}  |y|  \me^{ -  \frac{y^2}{2}}   \d  y \rightarrow 0 
\quad 
\mbox{as}
\quad
N \rightarrow \infty .
\end{align*}
\\
where $ N_\zs{1}= \big( N- (|s|+ \delta) |\mu|\big)/ \sigma_\zs{1}$, and $ s, \mu, \sigma_\zs{1}$ are fixed.
\\
Thus, for any  $ s, t$,
and $ u$, 
\begin{equation}\label{eec: FKm: lim}
\lim_\zs{\delta \rightarrow 0} \pmb{L}^*_\zs{\delta}( u)  =  0 .
\end{equation}
So, by the Lebesgue dominated theorem
$
 \int^T_\zs{t} \pmb{L}^*_\zs{\delta}( u) \d u \rightarrow 0.
$
\\
Hence \cref{Pr.sec:FKm.2}. \fdem

\subsection{Proof of \cref{sec: verif 1}}
From the optimal wealth process given in \cref{eq:dXt*} through It\^o formula we have that
$$
X^*_\zs{t}= x  \exp \bigg\{ \int^t_\zs{0} a^*(u) \d u \bigg\}  \cE_\zs{0, t}(b^*) ,
$$
where the function $ a^*$ and $ b^*$ are defined in \cref{eq:dXt*} and 
$$
 \cE_\zs{0, t}(b^*)= \exp \bigg\{ \int^t_\zs{0} b^*(u) \d W_\zs{u} - \frac{1}{2} \int^t_\zs{0} (b^*(u))^2 \d u \bigg\} .
 $$ 
 We will show \cref{sec: verif 1}  for $ \check{\delta} = 1+ (1- \gamma)/2 \gamma $,
 taking into account that 
 $$
  z(\varsigma, t) \le \c^{*} x^\gamma \exp\{ s^2 g(0)/2 \} \,.
  $$
  To this end it is sufficient to show that
 $$
 \sup_\zs{\tau \in \cM_\zs{t}} \E \Bigg( (X^*_\zs{\tau})^{\delta_\zs{1}}   \exp\bigg\{ \frac{\check{\delta}_\zs{1}}{2}  S_\zs{\tau}^2 \bigg\} 
  \bigg|  X_\zs{t}= x, S_\zs{t}= s \Bigg) < + \infty  ,
 $$
 where $\delta_\zs{1}= \gamma \check{\delta}=(1+\gamma)/2<1$ and $ \check{\delta}_\zs{1}= g(0) \check{\delta}$.
 Note that $ \big( \cE_\zs{t, u} \big)_\zs{u \ge t}$ is supermartingale and 
 $
 \E  \cE_\zs{t, \tau} (b^*) \le 1
 $ 
 for any stopping time $ \tau \in \cM$.
 Moreover, note that
 $$
 S_\zs{\tau}= \me^{- \varkappa (T-t)}  s+ \xi_\zs{t, \tau} 
 \quad
 \mbox{and}
 \quad
  \xi_\zs{t, \tau} =  \sigma  \me^{- \varkappa \tau} \int^\tau_\zs{t} \me^{ \varkappa u} \d W_\zs{u} .
 $$
Since
$|S_\zs{\tau}| \le |s|+ |\xi_\zs{t, \tau}|$,
  one needs to check that 
  $$
 \sup_\zs{\tau \in \cM_\zs{t}} \E \Big( (X^*_\zs{\tau})^{\delta_\zs{1}}   \exp\big\{  \check{\delta}_\zs{1} \xi_\zs{t, \tau}^2 \big\}  \Big) < + \infty .
 $$
 By  H\"older inequality we obtain that for $ p= (1+\delta_\zs{1})/2 \delta_\zs{1}$ and  $ q= (1+ \delta_\zs{1})/(1- \delta_\zs{1})$
 $$
 \E_\zs{\varsigma, t} (X^*_\zs{\tau})^{\delta_\zs{1}} \exp\big\{  \check{\delta}_\zs{1} \xi_\zs{t, \tau}^2 \big\}
					 \le 
 					\Big( \E_\zs{\varsigma, t} (X^*_\zs{\tau})^{ \delta_\zs{2}}   \Big)^{\frac{1}{p}}
  					\Big( \E_\zs{\varsigma, t} \exp\big\{  q \check{\delta}_\zs{1} \xi_\zs{t, \tau}^2 \big\}  \Big)^{\frac{1}{q}} ,
 $$
 where $\delta_\zs{2}= p \delta_\zs{1}=(1+\delta_\zs{1})/2<1$.
 Note that 
$$
 \E_\zs{\varsigma, t} (X^*_\zs{\tau})^{ \delta_\zs{2}} = x^{\delta_\zs{2}} \E \exp \bigg\{ \delta_\zs{2} \int^\tau_\zs{t} a^*(u) \d u \bigg\} \Big( \cE_\zs{t, \tau}(b^*) \Big)^{\delta_\zs{2}}.
 $$
 By H\"older inequality,  for  $ r = 1/ \delta_\zs{2}$
and~$q_\zs{1}= 1/(1- \delta_\zs{2})$  
\begin{align} 
\begin{split}
 \E_\zs{\varsigma, t} (X^*_\zs{\tau})^{ \delta_\zs{2}}
  		&\le
  x^{\delta_\zs{2}} 
 \bigg( 
 		\E \exp \bigg\{ 
 		\frac{\delta_\zs{2} }{1- \delta_\zs{2}}  \int^\tau_\zs{t}  a^*(u)  \d u
  					 \bigg\}
 \bigg)^{1- \delta_\zs{2}}
  \Big( \E_\zs{t, \tau}  \cE_\zs{t, \tau}( b^*) \Big)^{\delta_\zs{2}}
  		\\[2mm]
  		&\le
  x^{\delta_\zs{2}} 
 \bigg( 
 		\E \exp \bigg\{ 
 		\frac{\delta_\zs{2} }{1- \delta_\zs{2}}  \int^T_\zs{t}  |a^*(u)|  \d u
  					 \bigg\}
 \bigg)^{1- \delta_\zs{2}} .
 \end{split}
  \end{align}
 Moreover, note that
  \begin{align*}
  |a^{*}(t)|
  	  \le 
  \Big(g(0)+ \frac{\kappa_\zs{1}}{\sigma^2} \Big) \kappa_\zs{1} s^2 + \kappa_\zs{1} |s| \bold{B}_\zs{1}  + 1+ r 
  				 \le  				
  				 \kappa_\zs{2} s^2+ \c^{*} ,
\end{align*}
 where~$\c^{*} $ is some constant and $ \kappa_\zs{2}= \kappa^2_\zs{1}(1/\sigma^2 + 1/2) +  g(0) \kappa_\zs{1}$.
So, for some $\c^{*}>0$ 
$$
\int^T_t |a^*(u)| \d u 
		\le 
2 \kappa_\zs{2} \int^T_t \xi^2_\zs{t,u} \d u +\c^{*} .
$$
Let us show now that 
$$
\E \exp \bigg\{ \tilde{\kappa}_\zs{2} \int^T_t \xi^2_\zs{t, u} \d u  \bigg\} 
				< +\infty ,
$$
where
$\tilde{\kappa}_\zs{2}= 2 \delta_\zs{2} \kappa_\zs{2}/ (1- \delta_\zs{2}) 
				= \Big( 2(3+\gamma ) \kappa_\zs{1}^2 \big(1/\sigma^2+ 1/2+ g(0)/\kappa_\zs{1}\big) \Big)/(1-\gamma) $.
$$
\E \exp \bigg\{ \tilde{\kappa}_\zs{2} \Big( \int^T_t \xi^2_\zs{t, u} \d u \Big) \bigg\}
						=
						\sum^\infty_\zs{m=0} \frac{\tilde{\kappa}^m_\zs{2}}{m !} \E \Bigg( \int^T_t \xi^2_\zs{t, u} \d u \Bigg)^m <+\infty.
$$
Moreover, note  that in view of the H\"older inequality
$$
\E \Bigg( \int^T_t \xi^2_\zs{t, u} \d u \Bigg)^m 
							\le 
							(T-t)^{m-1}  \int^T_t \E  \xi^{2m}_\zs{t, u} \d u.
$$
Taking into account
that $\xi_\zs{t, u} \sim \cN(0, \int^u_\zs{t} \me^{ - 2 \varkappa (u-v)} \d v)$, we obtain that
$$
\E  \xi_\zs{t, u}^{2m } 
				=
				 (2m-1)!! \Big( \int^u_\zs{t} \me^{-2 \varkappa (u-v)} \d v \Big)^m 
				\le 
				\frac{m!}{\kappa^m}
$$
and
\begin{align*}
\E \bigg( \int^T_\zs{t} \xi_\zs{t, u}^2  \d u \bigg)^m
					 \le  
					m! \frac{T^m}{ \kappa^m}.
\end{align*}
Therefore,
$$
\E \exp \bigg\{ \check{\kappa}_\zs{2} \int^T_\zs{t} \xi_\zs{t, u}^2  \d u \bigg\}
					=
					\sum^{\infty}_\zs{m=0} \frac{\check{\kappa}_\zs{2}^m}{m!}  \E  \bigg( \int^T_\zs{t} \xi_\zs{t, u}^2  \d u \bigg)^m
					\le 
					 \sum^\infty_\zs{m=0} \bigg( \frac{\check{\kappa}_\zs{2}}{\kappa} T \bigg)^m .
$$
In view of the definition of $T_\zs{0}$ in \cref{eq: T0} we obtain that the condition 
$T < T_\zs{0}$ implies that $T< \kappa/ \check{\kappa}_\zs{2} $, i.e. this series is finite.
Moreover, by \cref{Props: Kab-Perg} , we have that
$
\E \xi_\zs{t, \tau} 
			\le 
			m ! (2 \sigma^2 T)^m 
$ 
for all $ m \ge 1$,
and for any $\tau \in \cM_\zs{t}$,
So
\begin{align*}
\E_\zs{\varsigma, t} \exp \{ q \check{\delta}_\zs{1} \xi_\zs{t, \tau}^2 \}
			&= 
1+ \sum_\zs{m = 1}^\infty \frac{(q \check{\delta}_\zs{1})^m }{m !}  \E \xi^m_\zs{t, \tau}
			\\
			& \le 
1+ \sum_\zs{m = 1}^\infty \frac{(q \check{\delta}_\zs{1})^m }{m !}  m! (2 \sigma^2 T)^m 		
			 \le 
1+ \sum_\zs{m = 1}^\infty ( 2  q \check{\delta}_\zs{1}  \sigma^2 T)^m 
				< + \infty ,
\end{align*}
for $T< 1/ 2 q \sigma^2 \check{\delta}_\zs{1}= \gamma (1-\gamma)/ (3+\gamma) (1+ \gamma) \sigma^2 g(0)$, which is true due to the condition
$T<T_\zs{0}$.
Hence \cref{sec: verif 1}. \fdem

\subsection{The smoothness properties for the function $\E Q(\eta^{s, t}_\zs{u}, u)$}
\label{E Q}
\begin{lemma}\label{Lem: Prob.sol}
For any bounded function $Q$ in $\cX$ and for $u>t$
$$
\left| \frac{\partial }{\partial s} \int^T_\zs{t} \E Q(\eta^{s, t}_\zs{u}, u)  \d u \right| 
					\le 
					Q^*_\zs{t} \frac{2 }{\sigma} \sqrt{\frac{2(T-t)}{ \pi }}  
$$
where 
$$
Q^*_\zs{t}  =  \sup_{\substack{ s \in \mathbb{R} \\[1mm] t  \leq  u  \leq  T  }} \left| Q(s,u) \right|.
$$ 
\end{lemma}
\begin{proof}
By Fubini theorem if a function $Q>0$ then 
$$
\frac{\partial}{ \partial s} \E \int^T_\zs{t} Q( \eta^{s, t}_\zs{u}, u)  \d u =  \frac{\partial}{ \partial s} \int^T_\zs{t}  \E Q(\eta^{s, t}_\zs{u}, u)  \d u
$$
As
$$
\E  Q(\eta^{s, t}_\zs{u}, u)  =  \frac{1}{\sigma_\zs{1}(u,t)} \int_{\mathbb{R}} Q(y, u) \varphi \Big( \frac{y- s\mu(u, t)}{\sigma_\zs{1}(u,t)}\Big)  \d y \,,
$$
where 
$$
\varphi(\theta) = \frac{1}{\sqrt{2 \pi}} \me^{-\frac{\theta^2}{2} } \quad \mbox{and} \quad \theta = \frac{z- x\mu(u,t)}{\sigma_\zs{1}(u, t)} \,.
$$
Then we have that, 
$$
\E  Q( \eta^{s, t}_\zs{u}, u) =  \frac{1}{\sqrt{2 \pi} \sigma_\zs{1}(u,t)} \int_{\mathbb{R}} Q(y, u) \exp \bigg\{ - \frac{(y- s\mu(u,t))^2}{2 \sigma_\zs{1}^2 (u,t)} \bigg\}  \d y \,.
$$
Thus by deriving the last expression with respect to $s$,
$$
\frac{\partial}{\partial s}\E Q( \eta^{s, t}_\zs{u}, u) =  \frac{\mu(u,t)}{\sqrt{2 \pi} \sigma_\zs{1}^3(u,t)} \int_{\mathbb{R}} Q(y, u) (y-s \mu(u,t))\exp \bigg\{ - \frac{(y- s \mu(u,t))^2}{2 \sigma_\zs{1}^2 (u,t)} \bigg\}  \d y \,.
$$
Then by letting 
$$
v = \frac{y-s \mu(u,t)}{ \sigma_\zs{1} (u,t)} \,.
$$
$$
 \frac{\partial}{\partial s}\E Q(\eta^{s, t}_\zs{u}, u)  =  \frac{\mu(u,t)}{\sqrt{2 \pi} \sigma_\zs{1}(u,t)} \int_{\bbr}  Q( s \mu(u,t)+ v \sigma_\zs{1}(u,t), u) v \me^{-v^2/2}  \d v.
$$
By taking the absolute value for both sides we get
\begin{align*}
| \frac{\partial}{\partial s}\E Q(\eta^{s, t}_\zs{u}, u) |  & \le  \frac{\mu(u,t)}{\sqrt{2 \pi} \sigma_\zs{1}(u,t)} \int_{\bbr}  | Q( s \mu(u,t)+ v \sigma_\zs{1}(u,t), u) | |v| \me^{-v^2/2}  \d v \,,
\\
& \le  Q^*_\zs{t}  \frac{\mu(u,t)}{\sqrt{2 \pi} \sigma_\zs{1}(u,t)} \int_{\bbr}   |v| \me^{-v^2/2}  \d v \,,
\\
& = Q_\zs{t}^* \sqrt{\frac{2}{\pi}}  \frac{ \mu(u,t)}{ \sigma(u,t) \big( \int_\zs{t}^u \mu^2(u,z) \d z \big)^{1/2}} \,,
\end{align*}
where $Q_\zs{t}^*=  \sup_\zs{y \in \bbr, u \ge t} |Q(y,u)|$.
Therefore, 
\begin{align*}
| \frac{\partial}{\partial s}\E Q(\eta^{s, t}_\zs{u}, u) |  & \le  Q_\zs{t}^*\sqrt{\frac{2}{\pi}}  \frac{ \mu(u,t)}{ \sigma_\zs{1}(u,t)} \,.
\end{align*}
Since the integral  
\begin{align*}
\int^T_\zs{t} \frac{\mu (u,t)}{\sigma_\zs{1}(u,t)}  \d u & =  \int^T_\zs{t} \frac{\me^{-\int^u_\zs{t} g_\zs{1}(v)  \d v}}{ \sigma \sqrt{\int^u_\zs{t} \me^{-2 \int^u_z g_\zs{1}(v)  \d v}  \d z}}  \d u 
							\\
							& \leq 
							 \int^T_\zs{t} \frac{1}{\sigma \sqrt{u-t}}  \d u
							 \;  = 
							  \frac{2}{\sigma}   \sqrt{T-t}.
\end{align*}
Therefore, 
\begin{align} \label{Apx: 4}
\begin{split}
\Bigg| \frac{\partial}{\partial s}\E Q(\eta^{s, t}_\zs{u}, u)\Bigg| 
							 \leq  Q_\zs{t}^* \sqrt{\frac{2}{ \pi}}  \frac{1}{ \sigma \sqrt{u-t}} \,.
\end{split}
\end{align}
Therefore by taking the integral from $t$ to $T$,
\begin{align} \label{Prob. sol.01}
\bigg| \frac{\partial}{\partial s}\int^T_\zs{t} \E  Q( \eta^{s, t}_\zs{u}, u)  \d u \bigg|  
						\le  Q^*_\zs{t}  \sqrt{\frac{2}{ \pi}}   \int^T_\zs{t} \frac{1}{ \sigma \sqrt{u-t}}  \d u
						\le Q^*_\zs{t} \frac{2 }{\sigma} \sqrt{\frac{2(T-t)}{ \pi }}     \,.
\end{align}
Hence \cref{Lem: Prob.sol}.
\end{proof}
\bigskip

We need the following Proposition  from \cite{KabanovPergamenshchikov2003}.
\begin{proposition}\label{Props: Kab-Perg}
Let $y$ be the solution to the following equation
\begin{equation} \label{eqn: lemma 1.1.1}
 \d y_\zs{t} =  f(y_\zs{t}, t)  \d t + G_\zs{t}  \d W_\zs{t}, \qquad y_\zs{0} = 0 .
\end{equation}
where  $y f(y, t) \le - \kappa y^2$ for all $y$ and for $\kappa >0$, and 
$$
\sup_\zs{0 \le t \le T} |G_\zs{t}| 
					\le 
					M,
					\quad
					\mbox{for all} 
					\quad
					M>0.				
$$
Then for any stopping time $0 \le \tau \le T$
$$
\E  |y_\zs{\tau}|^{2m}  \leq   m!  (2  M^2  T)^m .
$$
\end{proposition}

\section*{Acknowledgments}
 This research
  is partly supported by the Ministry of Education and Science of the Russian Federation in the framework of the research project no. 2.3208.2017/4.6
   and by  the RSF grant 17-11-01049 (National Research Tomsk State University).

\end{document}